\documentclass[11pt]{amsart}
\usepackage{epigamath}

\usepackage[english]{babel}

\numberwithin{equation}{section}

\usepackage[shortlabels]{enumitem}
\usepackage{tikz}
\usepackage{amssymb}

\newtheorem{thm}{Theorem}[section]
\newtheorem{lemma}[thm]{Lemma}
\newtheorem{proposition}[thm]{Proposition}

\theoremstyle{remark}

\theoremstyle{definition}

\newcommand{\QQ}{\mathbb{Q}}

\newcommand{\ii}{{\rm i}}
\newcommand{\CC}{\mathbb{C}}
\newcommand{\PP}{\mathbb{P}}

\newcommand{\lra}{\longrightarrow}

\setlist[enumerate,1]{label={\bf Step \arabic*.}, ref={\rm\arabic*}} 

\newcommand{\fake}{\operatorname{fake}}
\newcommand{\PSL}{\operatorname{PSL}}

\newenvironment{sizeddisplay}[1]
 {\par\nopagebreak#1\noindent\ignorespaces}
 {\nopagebreak\ignorespacesafterend}

\EpigaVolumeYear{7}{2023} \EpigaArticleNr{17} \ReceivedOn{September 21, 2021}
\InFinalFormOn{December 6, 2022}
\AcceptedOn{April 19, 2023}

\title{On equations of fake projective planes with automorphism group of order 21}
\titlemark{Fake projective planes with 21 automorphisms}

\author{Lev Borisov}
\address{Department of Mathematics, Rutgers University, Piscataway, NJ 08854, USA}
\email{borisov@math.rutgers.edu}

\authormark{L.~Borisov}

\AbstractInEnglish{We study Dolgachev elliptic surfaces with a double and a triple fiber and find explicit equations for two new pairs of fake projective planes with $21$ automorphisms. This includes, in particular, the fake projective plane discovered by J.~Keum.}



\acknowledgement{The author acknowledges the Office of Advanced Research Computing (OARC) at Rutgers, The State University of New Jersey for providing access to the Amarel cluster \cite{Amarel}
  and associated research computing resources that have contributed to the results reported here.}

\begin{document}


\maketitle

\begin{prelims}

\DisplayAbstractInEnglish

\bigskip

\DisplayKeyWords

\medskip

\DisplayMSCclass

\end{prelims}


\newpage

\setcounter{tocdepth}{1}

\tableofcontents


\section{Introduction}
The theory of  fake projective planes originated with the famous example of D. Mumford \cite{Mumford} of a surface of general type with the same Hodge numbers as the usual projective plane $\CC\PP^2$. By its  nature, the construction did not yield any explicit equations of it.
Over the subsequent decades, work by multiple authors  (see for instance \cite{AK,I,Keum.FPP,Keum.quot,KK,Klingler,PY,PY2}) produced additional examples and general results and culminated in the classification
of all fake projective planes by D. Cartwright and T. Steger \cite{CS11,CSlist}. These surfaces are classified as free quotients of the complex $2$-dimensional ball  $\mathbb B^2=\{(z_1,z_2),|z_1|^2+|z_2|^2<1\}$ by certain discrete arithmetic subgroups. There are exactly $50$ conjugate pairs of such surfaces, separated into $28$ classes. This extremely useful classification does not lead to any polynomial equations either since there are no known methods for constructing explicit automorphic forms for these groups.

Over the last several years, the author of this paper has been involved in multiple collaborations with the goal of discovering explicit polynomial equations that define fake projective planes and related surfaces; see \cite{BK19,by,BF20,BBF20}. This paper is a continuation of such efforts, originally aimed at finding the equations of Mumford's fake projective plane. While this goal is still elusive, we find equations of the fake projective plane constructed by J.~Keum in \cite{Keum.FPP}, which is commensurable to the Mumford's surface. We also find another interesting fake projective plane in the process. As always, new approaches had to be developed for the case at hand.

It is not surprising that most of the currently computed fake projective planes have nontrivial automorphism groups, as this provides some avenues for exploration, and this paper continues the trend. According to the classification of Cartwright and Steger, the maximum order of the automorphism group of a fake projective plane is $21$.
There are three conjugate pairs of fake projective planes with  automorphism group of this size, and in all three cases the group is the semi-direct product of a normal subgroup $C_7$ and a non-normal subgroup $C_3$. Specifically, here  are the planes, with their names in the CS classification and brief comments:
\begin{itemize}
\item {$(a=7$, $p=2$, $\emptyset$, $D_3 2_7)$} is the first example of the fake projective plane  for which explicit equations were found; see \cite{BK19}. There are two other conjugate pairs of fake projective planes in its class.
\item {$(a=7$, $p=2$, $\{7\}$, $D_3 2_7)$} is the surface constructed by Keum in \cite{Keum.FPP}. It has three more pairs of fake projective planes in its class, including Mumford's fake projective plane. 
\item {$(C20$, $p=2$, $\emptyset$, $D_3 2_7)$} is the last of the three surfaces; it does not seem to be implicated in any other construction. There are no other pairs in its  class.
\end{itemize}

In all three cases, the quotient of $\PP^2_{\fake}$ by the subgroup $C_7$ of its automorphism group has a minimal resolution $Y$ with rather peculiar geometry; see \cite{Keum.FPP,Keum.quot}. The quotient has three singular points of type $\frac 13(1,7)$, which are permuted by the residual $C_3$-action of the automorphism group of $\PP^2_{\fake}$. The minimum resolution $Y$ of $\PP^2_{\fake}/C_7$ has three disjoint chains of three lines with self-intersections $-3, -2, -2$ which we denote by
$$
S \relbar B  \relbar C, ~~S_1 \relbar B_1  \relbar C_1, ~~S_2 \relbar B_2 \relbar C_2.
$$
Here $\relbar$ indicates a transversal intersection point.
In addition, $Y$ is fibered over $\CC\PP^1$, with generic fibers of genus~$1$, two multiple fibers, three nodal fibers and one fiber of type $I_9$ (a ring of nine $\CC\PP^1$ with self-intersection $(-2)$ each) with components 
$$
A \relbar B  \relbar C \relbar A_1 \relbar B_1  \relbar C_1 \relbar A_2 \relbar B_2 \relbar C_2 \relbar A.
$$ 
The residual automorphism group $C_3$ preserves the fibration structure and acts by $A\to A_1\to A_2\to A$ and similarly for the $S$, $B$ and $C$ curves.

The multiplicities of the multiple fibers in the case of $(a=7,p=2,\emptyset,D_3 2_7)$ are $2$ and $4$, and they are $2$ and $3$ in the other two cases, which are the focus of this paper. In particular, in the cases of interest, the curves $S$, $S_1$ and $S_2$ are $6$-sections of the fibration. The two special fibers $3F_2$ and $2F_3$ have multiplicity $3$ and $2$, respectively. The reductions $F_2$ and $F_3$ are linearly equivalent to $2F$ and $3F$ with $F=K_Y$, and the generic fiber is equivalent to $6F$.

The first idea of this paper is to consider the ring
$$
\bigoplus_{a,b\geq 0} H^0(Y, \mathcal O(aF + bS)).
$$
It has a double grading, and we can derive a formula for the graded dimension
$$
\sum_{a,b\geq 0} \dim H^0(Y,\mathcal O(aF+bS)) \,t^a s^b = \frac {1+ 2s t^4 + 2s t^5  +s^2 t^9 }{(1-t^2)(1-t^3)(1-s)(1-st^3)}
$$
which is suggestive of a free module structure over the subring generated by the variables $u_0,u_1,v_1,v_2$ with weights
$(0,2),(0,3),(1,0),(1,3)$, respectively. The appropriate GIT quotient is a birational model $Y_0$ of $Y$ that collapses all of the curves that intersect $F$ and $S$ trivially.  These are the curves $C,B_1,C_1,B_2,C_2$, and  the image of the special fiber in $Y_0$ becomes
\begin{equation}\label{nodes}
A \relbar B \relbar *\relbar A_1\relbar **\relbar A_2 \relbar **\relbar A,
\end{equation}
so that the intersection point  of (the images of) $B$ and $A_1$ is a simple node and the intersections of $A_2$ with both $A$ and $A_1$ are $\frac 13(1,2)$ singularities.
The construction of the new fake projective planes then proceeds as follows.

\begin{enumerate}[wide]
\item\label{step1}
  We construct a nine-parameter family of $(2,3)$-Dolgachev surfaces with a rational $6$-section $S$. A~general member of this family has twelve distinct singular nodal fibers, in addition to the double and triple fibers. The defining equations of the family are nine quadrics of weights
$3\times (2,8)$, $3\times (2,9)$ and $3\times (2,10)$ in the variables of weight 
$$(0,2),(0,3),(1,0),(1,3),(1,4),(1,4),(1,5),(1,5).$$ 
The idea is to postulate the above free module structure and the weights of the quadratic relations and to use the associativity conditions of the ring to solve for the coefficients of the quadrics. It entails solving a system of over $1600$ equations with $92$ unknowns, which is done by an ad hoc method utilizing the Mathematica software system.

\item\label{step2}  We construct seven-, five- and two-parameter subfamilies with additional conditions on the special fiber. Respectively, we require the special fiber to contain a line, two disjoint lines, two disjoint lines with two nodes on one and one node on the other. In particular, a generic element
of the two-parameter family has special fiber 
$$
A \relbar B \relbar *\relbar A_1\relbar *\relbar A_2 \relbar *\relbar A
$$
in the sense that the intersection points of $B$ with  $A_1$ and  $A_2$ with $A$ and $A_1$ are nodes.

\item\label{step3} We find a finite-field reduction of the surface $Y_0$ by looking through the parameter choices over a finite field and checking whether the resulting surfaces have worse-than-nodal singularities at two special points on the curve $A_2$. The smallest prime for which we were able to find such surfaces was $79$.

\item\label{step4} We proceed by successively solving the (hard to write) conditions on being more singular at the intersection points of $A$-curves for parameters modulo powers of $79$. We then recognize the parameters as algebraic numbers and construct $Y_0$ over a number field of degree $12$. We make a coordinate change to 
realize $Y_0$ over the number field $\QQ(\sqrt {-7})$.

\item\label{step5} We study $Y_0$ to find its geometric features, such as the curves $S_1$ and $S_2$ and the birational action of $C_3$. We find the degree $7$ extension of the field of rational functions of $Y_0$ that gives $\PP^2_{\fake}$ and calculate the bicanonical linear system of the latter. We then realize $\PP^2_{\fake}$ as an intersection of $84$ cubic equations in $\CC\PP^9$,  following the blueprint of \cite{BK19}.

\item\label{step6} We identify the fake projective plane as $(C20,p=2,\emptyset,D_3 2_7)$ by exhibiting too many torsion line bundles for it to be $(a=7,p=2,\{7\},D_3 2_7)$. The method is to find non-reduced $C_3$-invariant elements of $|2K_{\PP^2_{\fake}}|$ modulo a prime (this time it is $29$) by an exhaustive search and then lift them to powers of the said prime and finally the algebraic numbers.  We also use one of the  torsion line bundles to pick a more natural basis of $H^0(\PP^2_{\fake},2K)$ so that its equations have smaller coefficients. Finally, we verify that it is indeed a fake projective plane, as in \cite{BK19}. 

\item\label{step7} In Step~\ref{step2}, one actually finds \emph{two different five-parameter families} of Dolgachev surfaces which contain two disjoint lines in the special fiber. Unfortunately, for the second family we were unable to reduce the number of parameters further by considering the condition of having nodes. However, we are still able to go through Steps~\ref{step3}--\ref{step5} in this case, by a brute force approach to the finite-field search.  By a process of elimination, we see that  the new pair of fake projective planes is the one constructed in \cite{Keum.FPP}. 
\end{enumerate}

The paper is organized as follows. Section~\ref{sec2} contains the first two steps of the construction. Section~\ref{sec3} contains Steps~\ref{step3} and~\ref{step4}. Section~\ref{sec4} describes Steps~\ref{step5} and~\ref{step6}. In Section~\ref{sec5} we discuss the last step. Finally,  in Section~\ref{sec6} we talk about the open problems associated with our construction.
We also have an appendix in which we put some equations that are too lengthy for the main body of the paper. However, many of the key formulas are far too large to even be included in the appendix. They are collected in \cite{Bdata} instead.

\subsection*{Acknowledgments} This project depended heavily on the use of the Mathematica software system, with certain steps performed in Magma, Macaulay2, PARI/GP and C; see \cite{Ma,Mag,M2,gp}. The author thanks Sai Kee Yeung for a useful comment on the first version of the paper.

\section{Families of (2,3) Dolgachev surfaces}\label{sec2}
We start by studying smooth projective surfaces $Y$ with a genus $1$ fibration $Y\to \CC\PP^1$ with the following properties: 
\begin{itemize}
\item The class of the general fiber is $6F$, where $F$ is the canonical class of $Y$. In particular, $F^2=0$.
\item There exist a double fiber $2F_3$ and a triple fiber $3F_2$. The classes of $F_2$ and $F_3$ are $2F$ and $3F$, respectively.
\item There is a rational $6$-section $S$ with $SF=1$ and  $S^2=-3$.
\item We have $p_g(Y)=q(Y)=0$.
\end{itemize}
Our motivation is that the minimal resolutions of the fake projective planes we are interested in satisfy the above;  see \cite{Keum.quot}.

As implied in the introduction, we first compute\footnote{Technically, we could just conjecture everything, with the justification for it being the final outcome, but it is worth proving what we can.} the graded dimension of the ring
$$
R=
\bigoplus_{a,b\geq 0} H^0(Y, \mathcal O(aF + bS))
$$
in a series of lemmas.

\begin{lemma}\label{b0}
The graded dimension of\, $R_0=\bigoplus_{a\geq 0} H^0(Y, \mathcal O(aF))$ is given by
$$
\sum_{a\geq 0} t^a  \dim H^0(Y, \mathcal O(aF))= \frac 1{(1-t^2)(1-t^3)}.
$$
\end{lemma}

\begin{proof}
The ring $R_0$ is freely generated by the elements $u_0\in H^0(Y,\mathcal O(2F))$ and $u_1\in H^0(Y,\mathcal O(3F))$ whose divisors are $F_2$ and $F_3$, respectively. Indeed, these generate a subring of $R_0$, and we will show that there are no other forms. The subring for $a=0\mod 6$ is isomorphic to the homogeneous ring of the $\CC\PP^1$ of the base of the fibration, so it is a polynomial ring in $u_0^3$ and $u_1^2$. 
If $a$ is odd, then  $\mathcal O(aF)$ restricts to a nontrivial bundle to $F_3$ (because the normal bundle of the double fiber is nontrivial), so all its global sections vanish on $F_3$, and the corresponding elements in $R_0$ are divisible by $u_1$. Similarly, if $a$ is not divisible by $3$, then all global sections of $\mathcal O(aF)$ vanish on $F_2$, and the elements are divisible by $u_0$. Together, these observations imply the result.
\end{proof}

\begin{lemma}\label{b1}
The dimension of\, $H^0(Y,\mathcal O(aF+S))$ is $1$\! for $a=0$ and  is $a-1$ for $a>0$.
\end{lemma}

\begin{proof}
For $0\leq a\leq 2$ we have $(aF+S)S<0$, so any section of $\mathcal O(aF+S)$ must vanish on $S$, and the statement follows from Lemma~\ref{b0}. For $a\geq 3$ we have $\chi(aF+S) = \frac 12(aF+S)((a-1)F+S)+1 = a-1$, so it suffices to show that the invertible sheaf $\mathcal O(aF+S)$ has no higher cohomology. The vanishing of $H^2(Y,\mathcal O(aF+S))$ for all integers $a$ is clear from Serre duality and the fact that every 
effective divisor on $Y$ must have a non-negative intersection with $F$. To see the vanishing of $H^1(Y,\mathcal O(aF+S))$, we run induction on $a$. Specifically, consider the long exact sequence.
\begin{align*}
&0\lra H^0(Y,\mathcal O((a-2)F+S)) \lra H^0(Y,\mathcal O(aF+S)) \lra H^0\left(Y,i_*\mathcal O_{F_2}(aF+S)\right)
\\
&\lra H^1(Y,\mathcal O((a-2)F+S)) \lra H^1(Y,\mathcal O(aF+S)) \lra H^1\left(Y,i_*\mathcal O_{F_2}(aF+S)\right).
\end{align*}
Since $(aF+S)F_2=2$, the last term is $0$, so it suffices to prove that  $H^1(Y,\mathcal O(aF+S))=0$ for $a\in {1,2}$, which follows from our computation of the global sections of these divisors.
\end{proof}

\begin{lemma}\label{ball}
  For  $b\geq 1$ and $a \geq 3b$ the dimension of\, $H^0(Y,\mathcal O(aF+bS))$
  is equal to
$\chi( \mathcal O(aF+bS))=$ $\frac 12(2ab-b - 3b^2 +2)$.
\end{lemma}

\begin{proof}
We will prove it by induction on $b$ with the base case provided by Lemma~\ref{b1}. As before, $\dim H^2(Y,\mathcal O(aF+bS))=
\dim H^0(Y,\mathcal O((1-a)F - bS))=0$, so the statement amounts to the relation $\dim H^1(Y,\mathcal O(aF+bS))=0$.
For the induction step, the short exact sequence 
$$
0\lra \mathcal O(aF+(b-1)S) \lra \mathcal O(aF+bS)) \lra i_{S*}\mathcal O(aF+bS)\lra 0
$$
leads to 
$$
\lra H^1(Y, \mathcal O(aF+(b-1)S) \lra H^1(Y, \mathcal O(aF+bS)) \lra H^1(S, i_S^* \mathcal O( aF+bS)).$$
The terms on the left and on the right are zero by the induction hypothesis and $(aF+bS)S=a-3b\geq 0$.
\end{proof}

The above lemmas allow us to compute the graded dimension of $R$.
\begin{proposition}\label{doublegraded}
$$\sum_{a,b\geq 0} \dim H^0(Y,\mathcal O(aF+bS)) \,s^b t^a = \frac {1+ 2s t^4 + 2st^5  +s^2 t^9 }{(1-t^2)(1-t^3)(1-s)(1-st^3)}.$$
\end{proposition}

\begin{proof}
We write $ \dim H^0(Y,\mathcal O(aF+bS)) =c_{a,b}$ to simplify notation. Since $(aF+bS)S<0$ implies that $S$ is a fixed component 
of $|aF+bS|$, we see that for $a<3b$ there holds 
$c_{a,b}=c_{a,b-1}$. Therefore, $\sum_{a,b\geq 0} c_{a,b} t^a s^b $ equals
$$
 \frac s{1-s}  \sum_{b \geq 0}
\left( c_{3b,b} t^{3b} + c_{3b+1,b} t^{3b+1}+ c_{3b+2,b} t^{3b+2} \right)s^b+
\sum_{b\geq 0}\sum_{a\geq 3b} c_{a,b} t^a s^b.
$$
We then use Lemmas~\ref{b0} and~\ref{ball} to write the above as
\begin{align*}
&\frac s{1-s}(1 + t^2)
+
\frac 1{(1-t^2)(1-t^3)}
\\&\;+
\sum_{b\geq 1} s^b\left( \frac s {1-s} \left( \frac 12(2(3b)b-b - 3b^2 +2) t^{3b}
+ \frac 12(2(3b+1)b-b - 3b^2 +2) t^{3b+1}\right.\right.
\\&\hphantom{\sum_{b\geq 1} s^b\Big(\frac s {1-s}\frac12}\;\;\;\;+ \frac 12(2(3b+2)b-b - 3b^2 +2) t^{3b+2} \Big)+  \sum_{a\geq 3b} \frac 12(2ab-b - 3b^2 +2) t^a \Big), 
\end{align*}
which is then easily computed by Mathematica.\footnote{It is, of course, computable by hand, but this seems to be a fool's errand given the subsequent use of various software.}
\end{proof}

As a consequence of Proposition~\ref{doublegraded}, we conjecture that $R$ has the structure of a rank $6$ graded free module over the ring 
$$
\CC[u_0,u_1,v_1,v_2]
$$
generated by the sections $u_0$, $u_1$, $v_1$ and $v_2$ of $\mathcal O(2F)$,  $\mathcal O(3F)$,  $\mathcal O(S)$ and  $\mathcal O(3F+S)$, respectively. Note that while $u_0,u_1,v_1$ are defined uniquely up to scaling, the section $v_2$ can also be changed by adding a scalar multiple of $u_1 v_1$. We denote the generators of the module at weights $(1,4)$ and $(1,5)$ by $v_3,v_4,v_5,v_6$. We conjecture that the generator at degree $(9,2)$ is equal to $v_3 v_5$.  It can likely be proved that the ring $R$ is Gorenstein, which would then imply that such a choice is possible, 
but we just take it as a sensible guess.

We will consider polynomial relations on $u_0,\ldots,v_6$. By looking at the graded dimension of $R$, we observe that these variables must satisfy
three linearly independent relations of degree $(2,8)$, three linearly independent relations of degree $(2,9)$ and three relations of degree $(2,10)$, which are linearly independent together with $u_0$ times the $(2,8)$ relations. 
We will refer to these relations as quadrics (in $v_i$). 
In general, genus $1$ curves of  degree $6$ in $\PP^5$ can be written as an intersection of nine degree $2$ polynomials, and we expect that for most values of $u_0$ and $u_1$, these quadrics describe the corresponding fiber of $Y\to \CC\PP^1$.

We then set up the possible quadrics with undetermined coefficients, taking care to undo multiple symmetries of the construction. For example,
we make sure that the coefficients of the $(2,8)$ quadrics in $v_3^2, v_3 v_4, v_4^2$ are the standard basis vectors, and similarly for the $(2,9)$ and $(2,10)$ quadrics. We assume that $(v_3,v_4)$ and $(v_5,v_6)$ are dual to each other in the socle pairing of $R/\langle u_0,u_1,v_1,v_2\rangle$.
We can also make additional assumptions on the coefficients in view of the possible changes of $v_i$ such as $v_2\mapsto\alpha v_2 +\beta u_1 v_1$. 
The only such assumptions which appear crucial to the success of the method are that  the coefficients of the first $(2,8)$ quadric at $v_3 v_1 u_0^2$ an $v_4 v_1 u_0^2$ are $1$ and $0$, the coefficients of the second quadric
at these monomials are both $0$, and the coefficients of the third quadric are $0$ and $1$, respectively. It can be shown that a generic collection of the quadrics can be manipulated into this form by appropriate linear changes of $v_3,v_4$ together with adding multiples of $ u_0^2 v_1$ to them. However, there are six ways of doing so, which means that while we expect $Y_0$ to be defined over a quadratic imaginary field, we cannot expect the coefficients of these quadrics to be this simple.

Once the quadrics are written down, one can compute the multiplication table for the generators of $R$ as a free module over $\CC[u_0,u_1,v_1,v_2]$ and then set up the associativity relations as equations on the coefficients of the quadrics; see \cite[DolgachevSurfaces.nb]{Bdata}.
The associativity relations led to over $1600$ equations in $92$ variables. Fortunately, some of these equations were quite simple, but still the task appeared daunting. We were able to use the Mathematica ``Solve'' command to eventually reduce to nine free parameters. The basic idea was to try to solve the easier equations first, with either byte count or the number of terms used as measure of complexity. The drawback of this technique is that one can at best hope to recover rational parameterizations by a subset of the set of variables, but we were fortunate in this case. The resulting equations are presented in Section~\ref{9param} in the appendix. The variables $d_i$ are the parameters, and the variables $u_0, u_1, v_1,\ldots, v_6$ are the coordinates.

Having found the equations of a nine-parameter family, we then tried to impose additional geometric conditions on it.
We postulated without loss of generality that the $I_9$ fiber of $Y$ occurs at $u_0^3=u_1^2$ (or simply $u_0=u_1=1$ if one wants to dehomogenize). We know from the work of Keum \cite{Keum.quot} that possible intersections of $S$ with $(A,A_1,A_2)$ fall into cases $(2,2,1)$ and $(1,1,3)$, so in either setting, the image of the $I_9$ fiber in $Y_0$ should have two disjoint lines. We first solved for the condition that the equations in Section~\ref{9param} vanish on one line, by parameterizing the said line. This gave a seven-parameter family, and we used its explicit description to put in a condition of having two such lines, taking care not to get the degenerate cases where the lines intersect. Much trial and error was involved, and even in the best case, the computations took a long time. Due to an unfortunate mistake on the author's part, some of the intermediate steps were lost; however, the final five-parameter answer survived and is explicitly written in \cite[DolgachevSurfaces.nb]{Bdata}.

The general member of this five-parameter family has a double and a triple fiber at $u_1=0$ and $u_0=0$, respectively, and a special fiber at $u_0^3=u_1^2$ which has the following configuration of curves $A$, $B$, $A_1$, $A_2$ and their intersection points $p_1,\ldots, p_4$: 

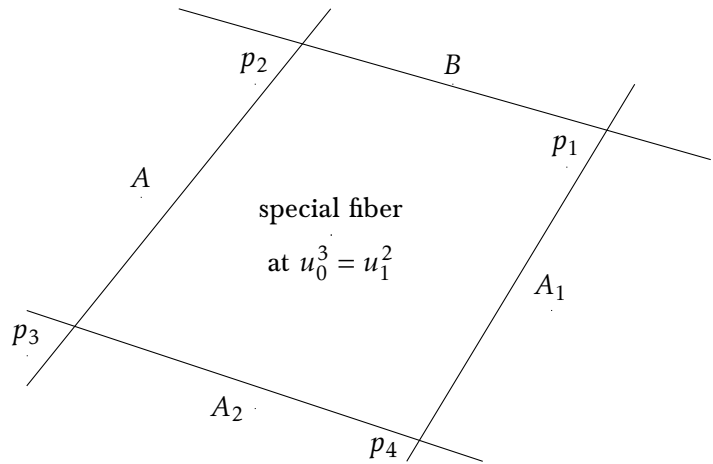
\begin{figure}[h!]
  \begin{center}
\begin{tikzpicture}
\draw (-2,-1) -- ( 2, 4);
\draw (0,4) -- (7,2);
\draw (6,3) -- (3,-2);
\draw (4,-2) -- (-2,0);
\draw (2,1)--(2,1 ) node[above]{special fiber};
\draw (2,1)--(2,1 ) node[below]{at $u_0^3=u_1^2$};
\draw (-2,-.6)--(-2,-.6) node[above]{$p_3$};
\draw (1,3)--(1,3) node[above]{$p_2$};
\draw (5.1,1.9)--(5.1,1.9 ) node[above]{$p_1$};
\draw (2.7,-2.1)--(2.7,-2.1 ) node[above]{$p_4$};
\draw (-.5,1.5)--(-.5,1.5) node[above]{$A$};
\draw (3.6,3.0)--(3.6,3.0) node[above]{$B$};
\draw (1.0,-1.3)--(1.0,-1.3) node[left]{$A_2$};
\draw (4.9,0)--(4.9,0) node[above]{$A_1$};
\end{tikzpicture}
  \end{center}
  \end{figure}

Here,  $A_2$ and $B$ are the two disjoint lines in the special fiber, and  $A$ and $A_1$ are both of degree $2$. 

Our next step was to impose the conditions that $Y_0$ is singular at $p_1$, $p_3$ and $p_4$, as would be expected by the geometry of the surface.
One immediate technical difficulty was that the points $p_1$ and $p_2$ were not defined over the parameter space. However, we were able to make a change of variables, similar to a rational parameterizations of a plane conic, to resolve this difficulty. We also made a simple coordinate change to simplify the formulas somewhat; see \cite[DolgachevSurfaces.nb]{Bdata} for details. 

The conditions of being singular at each $p_i$ were computed as follows. We looked at the Jacobian matrix of the quadrics at $p_i$ and computed its minors. We then computed the greatest common factor of these minors, which left us with large, but manageable expressions in the parameters $t_1,\ldots, t_5$. An \textit{ad hoc} manipulation of the equations and the parameters allowed us to replace $t_i$ with $s_1,\ldots,s_5$ and then solve for $s_1$ and $s_5$.  The polynomial equation on the remaining parameters $s_2$, $s_3$ and $s_4$ has several factors, with all but one of them
leading to undesirable degeneracies. We were left with a polynomial with integer coefficients of degree $3$, $13$ and $12$ in $s_2$, $s_3$ and $s_4$, respectively. It is about $130$ Kb long in its expanded form and thus is not worth trying to write down explicitly in this paper; see \cite[DolgachevSurfaces.nb]{Bdata}. It appears that the resulting $2$-dimensional parameter space is not rational, but it was simple enough for our purposes.

\section{Finding the special Dolgachev surface}\label{sec3}
This section describes how we found the surface $Y_0$. We expected that $Y_0$ has $\frac 13(1,2)$ singularities at $p_3$ and $p_4$, and it would be reasonable to try to encode these in terms of $s_2$, $s_3$ and $s_4$ subject to the large polynomial equation.
Unfortunately, this direct approach was not computationally feasible, as equations ballooned to hundreds of Mb in length. So we used the alternative method of finite-field reduction. Specifically, we wanted to find a reduction of $Y_0$ modulo a prime $p$ and then lift it first to $p$-adics and then to a number field. 

We looked at various primes $p$ and parameters $(s_2,s_3,s_4)\in\mathbb F_p^3$, with $p$ chosen so that $\sqrt{-7}$ exists in $\mathbb F_p$ as it appeared likely that we would need this in the field of definition. For each triple of parameters we first checked whether it fits the polynomial relation that defines the parameter space. Then we used Magma to compute the structure of singularities at $p_3$ and $p_4$ (we have lengthy but explicit formulas for $p_i$ in terms of the parameters). If both singularities looked right, then we took note of the values of the parameters. 
For one reason or another, we had to go to $p=79$ before any suitable examples were found! However, at $p=79$ there were six possible solutions, see \cite[Search79]{Bdata}: 
$$
(s_2,s_3,s_4) \in 
\{ (14, 47, 52),
(15, 65, 27),
 (19, 32, 14),
(44, 14, 32), (58, 27, 65), (72, 52, 47)\}.
$$
 
It was then an interesting challenge to lift these solutions to powers of $79$. By solving for $v_5$ and $v_6$, one can reduce the problem to codimension $2$. Then it is possible to encode the condition of having worse-than-nodal singularity as a certain polynomial in terms of first and second derivatives of the defining equations. Then we endeavored to have these polynomials produce values that are zero modulo higher and higher powers of $79$, as we are adjusting our parameters $s_i$ to a more accurate $p$-adic approximation. Of course, we also have to keep the defining polynomial of the family zero to the appropriate power of $79$. The resulting code is messy but not particularly slow, and we were able to compute the values of $s_i$ up to $79^{101}$ in a reasonable amount of time.  

Once a $p$-adic approximation was found, it was then routine to find ``simple'' algebraic numbers that give these parameters $s_i$.
As in \cite{BF20,BBF20}, we used a lattice reduction algorithm to find a small linear combination of powers of $s_i$ modulo $79^{101}$ and $79^{101}$ itself. This suggested that these are algebraic numbers of degree $12$. By looking at the standard polynomial, we indeed found $12$ possible triples $(s_2,s_3,s_4)$. Perhaps not surprisingly, our finite-field search only picked up the six cases that correspond to the simple roots of the reduction of the defining polynomial of $s_2$ 
\begin{align*}
  &1048576 + 9633792 s_2 + 47179776 s_2^2 + 156022272 s_2^3 + 376708864 s_2^4 \\
  &  +  693988960 s_2^5 + 1003433368 s_2^6  + 1148276192 s_2^7 + 1023247890 s_2^8 \\
&  +  681835980 s_2^9 + 317640295 s_2^{10} + 91989513 s_2^{11} + 12492403 s_2^{12}
 \end{align*}
modulo $79$.

Since we expected that $Y_0$ can be defined over a quadratic field, we wanted to find a coordinate change in $v_i$ to see it. Specifically, we aimed to get the points $p_i$ to have simple coordinates. As it is a lot faster in Mathematica, this was done numerically, and then the coefficients were approximated by algebraic numbers. The resulting equations  are listed in Section~\ref{Y0} of the appendix, where we use $w_i$ to denote new variables, with $u_0$ and $u_1$ not affected by the coordinate change.

\section{Constructing and identifying the fake projective plane}\label{sec4}
This section describes how we found the fake projective plane (up to conjugation) and identified it as $(C20,p=2,\emptyset,D_3 2_7)$.
The relevant calculations are contained in the Mathematica file SevenfoldCover.nb and Magma files Torsion and CheckSmoothness, as well as Macaulay2 file CheckFPP in \cite{Bdata}.

We know that $S$ is cut out by $w_1=0$, but it takes a bit of effort to find equations of $S_1$ and $S_2$.
The first idea is that since $(4F+S)S=4-3=1$, the sections $w_3$ and $w_4$ are linear on $S$. Thus, points on $S$ can be parameterized by 
$t=\frac {w_4}{w_3}$. This is true for any member of the $9$-dimensional family of Dolgachev surfaces; see Section~\ref{9param}. So we know both the equation and the parameterizations of $S$, and we would like to do the same for $S_1$ and $S_2$.

It follows from intersection form considerations \cite[IntersectionForms.nb]{Bdata} that $S_2$ is cut out by an equation of bidegree $(1,8)$
which also passes through $A$ and $A_2$ with multiplicities $1$ and $3$, respectively. We pick multiple points on $A$ and $A_2$ and encode the relevant conditions on the $(1,8)$ polynomial to find it to be 
\begin{equation}\label{S2}
  \begin{aligned}
u_0^4& w_1 + 
\frac  1{149} (-124 - 9 \,\ii \sqrt{7}) u_0 u_1^2 w_1 + \frac 1{8344} \left(-553 + 
    509 \,\ii \sqrt{7}\right) u_0 u_1 w_2   \\
    &+
 \frac 1{149} \left(20 - 37 \,\ii \sqrt{7}\right) u_0^2 w_3 + 
 \frac 1{596} \left(-155 + 175\, \ii \sqrt 7\right) u_0^2 w_4 
 \\
& + \frac 1{8344}\left(3255 + 
    1093 \,\ii \sqrt 7\right) u_1 w_5 + \frac 1{2086}\left(-763 - 369\,\ii \sqrt 7 \right) u_1 w_6.
  \end{aligned}
  \end{equation}
We also needed to compute the parametric equation of $S_2$. After some of the more straightforward approaches failed to finish, we did the following. For a given value of the parameter $t$ on $S$ we can compute the corresponding point $p(t)$ on $S$ and six points of $S_2$ in the same fiber of the genus $1$ fibration. One of these points is the image $p_1(t)$ of $p(t)$ under the predicted order $3$ birational automorphism of $Y_0$.  If $t$ is an integer, then this point on $S_2$ has to be defined over $\QQ(\ii\sqrt 7)$, which allows us to distinguish $p_1(t)$ from the five other points. After the normalization $w_1=1$, the coordinates of $p_1(t)$ are polynomial in $t$ of degree at most $5$, and we use several values of $t$ to find their coefficients.
The resulting parameterization is given in Section~\ref{S2param} in the appendix.

One could follow a similar approach to find $S_1$, which has a degree $(1,10)$ equation vanish on it, but we chose to use the group law of the elliptic fibers instead. Indeed, we know that for a given value of $t$ the points $p(t)$, $p_1(t)$ and $p_2(t)$ satisfy $p_1(t)+p(t)=2p_1(t)$ under the group law of the fiber with any choice of origin. This allowed us to find $p_1(t)$ and then compute the parameterization; see Section~\ref{S1param}. We then found the degree $(1,10)$ polynomial 
\begin{equation}\label{S1}
  \begin{aligned}
u_0^5& w_1 - u_0^2 u_1^2 w_1 
+ 
\frac 1{7112}\left(-1897 - 3 \,\ii\sqrt 7\right) u_0^2 u_1 w_2
+ 
 \frac 1{254} \left(-3 + 17 \,\ii\sqrt 7\right) u_0^3 w_3 \\
 &+
 \frac 1{254} \left(39 + 33 \,\ii\sqrt 7\right) u_1^2 w_3 + 
 \frac 1{508} \left(-71 - 21 \,\ii\sqrt 7\right) u_0^3 w_4 + 
 \frac 1{254} \left(-71 - 21 \,\ii\sqrt 7\right) u_1^2 w_4
 \\
  &+ 
 \frac 1{7112}\left(6559 + 
    509 \,\ii\sqrt 7\right) u_0 u_1 w_5 
 +\frac 1{1778} \left(-1435 - 
    81 \,\ii\sqrt 7\right) u_0 u_1 w_6
\end{aligned}\end{equation}
that vanishes on $S_1$. 

We then computed the order $3$ birational automorphism by postulating its general form and using the parameterization of $S$, $S_1$ and $S_2$ to find the coefficients. It can be viewed as an order $3$ automorphism of a genus $1$ curve defined over $\CC(u_0,u_1)$. The formulas for the automorphism are in Section~\ref{Auto}.

We were then ready to compute the equations of the sevenfold cover of $Y$. The method was essentially the same as that of \cite{BK19}, but we briefly describe it here for the benefit of the reader.
We know that the sevenfold Galois cover $\PP^2_{\fake}\to Y$ is ramified at $S,B,C,S_1,B_1,C_1,S_2,B_2,C_2$. We consider the rational function
\begin{equation}\label{g7}
\frac {f_{1,10} f_{1,8}^2}{u_0^7 \left(u_0^3 - u_1^2\right)^2 w_1^3},
\end{equation}
where $f_{1,8}$ and $f_{1,10}$ are given in \eqref{S2} and \eqref{S1}, respectively. Its  divisor is supported on the special curves, with all but the nine $S,\ldots,C_2$ curves above having multiplicity divisible by $7$. So the function field of $\PP^2_{\fake}$ is obtained from that of $Y$ by adding the seventh root of the function \eqref{g7}. We then compute ten sections of the bicanonical linear system on  $\PP^2_{\fake}$ by projecting them to $Y$ and realizing them as sections of various line bundles on $Y_0$. Care is taken to have the action of the automorphism group look nice in these coordinates. Specifically, we have ten variables $P_0,\ldots,P_9$ on which the order $7$ element acts by 
$$
g_7(P_0,\ldots,P_9) = \left(P_0, \zeta_7 P_1,  \zeta_7^2 P_2,  \zeta_7^4 P_3, \zeta^3 P_4, 
\zeta^6 P_5 ,\zeta^5 P_6, \zeta^3 P_7, \zeta^6 P_8,\zeta^5 P_9\right)
$$
and the order $3$ element acts by 
$$
g_3(P_0,\ldots,P_9) = (P_0,P_2,P_3,P_1,P_5,P_6,P_4,P_8,P_9,P_7).
$$
See \cite[SevenfoldCover.nb]{Bdata} for more details.

Having constructed a fake projective plane with $21$ automorphisms, we want to try to identify it in accordance with the Cartwright--Steger classification. Since the Dolgachev surface it is built from has a double and a triple fiber, there are two possibilities:  Keum's surface 
$(a=7,p=2,\{7\},D_3 2_7)$ and $(C20,p=2,\emptyset,D_3 2_7)$. In the former case, the torsion in the Picard group is $C_2^3$, with the automorphism group  $G$ acting transitively on seven nontrivial elements. Therefore, for every cyclic subgroup of order $3$ in $G$ there should be a unique $G$-invariant torsion element on $\PP^2_{\fake}$. This naturally lead to us trying to construct such torsion classes.

If $L$ is a torsion class in the Picard group of a fake projective plane, then $H^0(\PP^2_{\fake}, K+L)$ is $1$-dimensional; see for example \cite{GKS}. If $L$ is furthermore a $2$-torsion element, then the square of the corresponding section is in $H^0(\PP^2_{\fake}, 2K)$. In the other direction, $2$-torsion elements of the Picard group can be constructed from sections of $2K$ which give non-reduced curves. The additional condition
of $C_3$ invariance means that we should consider sections
\begin{equation}\label{sec}
a_0 P_0 + a_1 (P_1+P_2+P_3) + a_2 (P_4+P_5+P_6) + a_3(P_7+P_8+P_9)
\end{equation}
up to scaling. 

We looked at the finite-field reduction of $\PP^2_{\fake}$ to look for such $a_i$. Specifically, $p=29$ was the smallest prime that both had $\sqrt{- 7}$ in it and had the expected Hilbert polynomial of the reduction of $\PP^2_{\fake}$. We then ran through all possible $a_i$ with $a_0=1$ and used Magma to check if the resulting scheme is non-reduced; see \cite[Torsion]{Bdata}.  We got three solutions, which was already a likely indicator that the $\PP^2_{\fake}$ is not $(a=7,p=2,\{7\},D_3 2_7)$ but rather is  $(C20,p=2,\emptyset,D_3 2_7)$, which has a larger torsion subgroup $C_2^6$ in its Picard group. However, we needed to ascertain it by lifting to characteristic zero. 

The lifting procedure was pretty typical of such approach: we lifted the $a_i$ in \eqref{sec} to modulo $p^k$ for large enough $k$ and then guessed the algebraic numbers that could give these reductions. It is worth mentioning how exactly we did the lifting. We found via Magma several points on the non-reduced linear cut of the reduction of $\PP^2_{\fake}$ modulo $p$; then at each point we found two linearly independent tangent vectors in $\PP^9(\mathbb F_{29})$ which are orthogonal to the gradients of all $84$ cubic polynomials and the linear polynomial \eqref{sec}. We lifted these vectors to small integers and then were successively adjusting them so that the aforementioned orthogonality held modulo higher powers of $p$. At each stage we modified the points, the tangent vectors and the polynomial \eqref{sec} to the next power of $p$; this amounted to solving a system of linear equations modulo $p$, which was not time-consuming. Afterwards, we identified the corresponding algebraic numbers.

One of the solutions was the linear polynomial
\begin{equation}\label{cut}
  \begin{aligned}
 & P_0 + \tfrac 12\left(1 + \,\ii \sqrt 7\right) (P_1 + P_2 + P_3) +
  \left(-122 + 2 \,\ii \sqrt 7\right) (P_4 + 
    P_5 + P_6)
    \\
  &   + \tfrac 17 \left(84 - 4\, \ii \sqrt 7\right) (P_7 + P_8 + P_9),
\end{aligned}
\end{equation}
and the other two were defined over a degree $4$ number field. For each of these equations we then verified that the resulting cuts are non-reduced.\footnote{This verification was only done numerically rather than symbolically, but it is sufficiently convincing for our purposes. An interested reader is welcome to try their hand at verifying it symbolically.}
We note that the resulting $2$-torsion line bundles are not $C_7$-invariant, so we have established at least $21$ nontrivial $2$-torsion elements of the Picard group, which indicates that our surface is (up to complex conjugation)   $(C20,p=2,\emptyset,D_3 2_7)$.

Last but not least, we used the description of the above non-reduced linear cut to find a more pleasant basis for $H^0(\PP^2_{\fake},2K)$; see \cite[SevenfoldCover.nb]{Bdata}.
Namely, we used a coordinate change from $P_i$ to $Q_i$  in which the aforementioned non-reduced cut  \eqref{cut} becomes
$$Q_0+Q_1+Q_2+Q_3+Q_7+Q_8+Q_9.$$
This allowed us to find a model of it with smaller coefficients of the $84$ equations, which we recorded in
\cite[EqsFPPwithrrQ]{Bdata}. We then went through the verification procedure developed in \cite{BK19} to make sure that the scheme cut out by the $84$ cubic relations in $10$ variables is indeed a fake projective plane in its bicanonical embedding; see \cite[CheckFPP, CheckSmoothness]{Bdata}. The smoothness calculation was performed by looking at three specific minors which are nonzero at the $C_7$ invariant points of $\PP^2_{\fake}$. Each minor took a few hours to compute on our hardware.

\section{Constructing the fake projective plane of Keum}\label{sec5}
While looking for Dolgachev surfaces with two disjoint lines, we have found a five-parameter subfamily of Dolgachev surfaces with two disjoint lines in the special fiber with the intersections of $S$ with $A$, $A_1$ and $A_2$ given by $(1,1,3)$. This makes it is provably different from the family in Step~\ref{step2}.
This family was noticeably harder to work with.  It was not even entirely straightforward to find the intersection points of the components of the special fiber. In particular, the conditions of being singular at three of these points were far too complicated to simplify. The details are in \cite[SecondFamily.nb]{Bdata}

To find a member of this family with additional singularities \eqref{nodes}, we used a brute force approach. Namely, we ran a search over the parameter space $\mathbb F_p^5$ for small $p$, computed the intersection points of the components of the special fiber and checked to see if the tangent space at the three points of interest had the correct dimension. This was too time-consuming in Mathematica, so we used lower-level languages. We first tried PARI/GP and then eventually C; see~\cite[finitefieldnodes.c]{Bdata}. Clearly, this is very easy to parallelize, and we ran multiple computations at a time on the Amarel cluster \cite{Amarel}. This allowed us to reduce the set of possible parameters to roughly $p^2$. For each of these we used Mathematica to check if there are worse-than-nodal singularities at the two expected points.

The first successful prime was $p=53$. We proceeded to lift the solution to $p$-adics, as described in Section~\ref{sec3}. We computed the sevenfold cover along the lines of Section~\ref{sec4}; see \cite[FindingKeumFPP.nb]{Bdata}. We do not write down the parametric equations or the $C_3$ symmetry, but do list the equations in Section~\ref{Y0Keum} for the record; they are also available 
at  \cite[QuadricsWforKeum]{Bdata}. The field extension needed to construct the cyclic cover was obtained by adding the seventh root of 
\begin{align*}
&\left(4 u_0^4 w_1 - 4 u_0 u_1^2 w_1 + \left(18  - 6 \ii \sqrt 7\right)u_0 u_1 w_2 - 
 \left(3  - \ii \sqrt{7}\right)u_0^2 w_3 + \left(5  + 3 \ii \sqrt 7 \right)u_0^2 w_4\right. 
 \\
 &
 \phantom{(}+\left.\left(-1- 3 \ii \sqrt 7 \right)u_1 w_5 +\left(- 19 - 7 \ii \sqrt 7\right)  u_1 w_6\right)\;\times \\
& \left (
 -66 u_0^2 u_1 w_2 +\left( 8  + 4 \ii \sqrt 7 \right)u_0^3 w_3+\left(- 2 +  10 \ii \sqrt 7 \right)u_1^2 w_3 
 +\left(- 31  + \ii \sqrt 7 \right) u_0^3 w_4 \right. 
 \\
 &
 \phantom{(}
 +\left(- 34 -  6 \ii \sqrt 7 \right) u_1^2 w_4
 +\left.\left(- 17 - 3 \ii \sqrt 7  \right)u_0 u_1 w_5 + 
 \left(142 - 6 \ii \sqrt 7\right) u_0 u_1 w_6
 \right)^2\big/\left(u_0^{14} w_1^3\right)
\end{align*}

We verified that the resulting equations give a fake projective plane by the same method. 
By elimination, it must be the one constructed by Keum in \cite{Keum.FPP}. We did not try to simplify the equations of the Keum's fake projective plane using the torsion elements in its Picard group, even though it must be possible to do. The plan is to have an undergraduate researcher to try to use our approach to simplify equations of this and other known fake projective planes.

\section{Open problems}\label{sec6}
The construction of this paper raises several  interesting questions. 

One can try to apply our approach to computing order $2$ elements of the Picard group of the  fake projective plane $(C20,p=2,\emptyset,D_3 2_7)$ to other fake projective planes. It is important to have at least a $C_3$ symmetry; otherwise the finite-field search may take too long. One can even try to do it with $C_3$-invariant elements of higher order by looking for reducible linear cuts in the bicanonical embedding, although it is less clear how to do the lifting in this case.



One can try to find a nicer set of variables so that the equations of the two special Dolgachev surfaces $Y_0$ are simplified. 

It is hoped that the knowledge of Keum's fake projective plane will allow one to find explicit equations of Mumford's fake projective plane. This appears to boil down to finding a degree $7$ non-Galois cover of the quotient of $\PP^2_{\rm Keum}$ by its  $C_3$ automorphism group, with the Galois group of the corresponding splitting field being $\PSL(2,7)$. Perhaps the key to this is understanding such covers for the fibers of 
the map $Y\to \PP^1$. In fact, the triple fiber leads one to look for such degree $7$ covers $C\to E$ of a genus $7$ curve $C$ which are ramified over two explicitly known points on an explicitly known genus $1$ curve $E$. This is currently work in progress.

\renewcommand\thesection{\Alph{section}}
\setcounter{section}{0}

\section*{Appendix} 
\addcontentsline{toc}{section}{Appendix}
\refstepcounter{section}


\subsection{Equations of the nine-parameter family of Dolgachev surfaces}\label{9param}

\allowdisplaybreaks

\begin{sizeddisplay}{\small}
\begin{align*}
\bullet &\;d_2 d_5 d_9 u_0 u_1^2 v_1^2 + d_2 d_3 d_9^2 u_0 u_1^2 v_1^2 - 
  d_2 d_6 d_9^2 u_0 u_1^2 v_1^2 + d_3 d_9 u_0 u_1 v_1 v_2 
  + d_2 d_4 d_9 u_0 u_1 v_1 v_2
\\&\; +
   d_2 d_5 d_9 u_0 u_1 v_1 v_2 + 2 d_2 d_3 d_9^2 u_0 u_1 v_1 v_2 - 
  2 d_2 d_6 d_9^2 u_0 u_1 v_1 v_2 + d_3 d_9 u_0 v_2^2 + d_2 d_4 d_9 u_0 v_2^2 
\\&\;+ 
  d_2 d_3 d_9^2 u_0 v_2^2 - d_2 d_6 d_9^2 u_0 v_2^2 + d_2 u_0^2 v_1 v_3 + d_2 v_3^2 - 
  d_2 d_9 u_1 v_1 v_5 - v_2 v_5 - d_2 d_9 v_2 v_5 
\\&\;+ d_2 d_9 u_1 v_1 v_6 + 
d_2 d_9 v_2 v_6
\\ \bullet&\;
-d_2 d_5 d_9 u_0 u_1^2 v_1^2 - d_2 d_3 d_8 d_9 u_0 u_1^2 v_1^2 + 
  d_2 d_6 d_8 d_9 u_0 u_1^2 v_1^2 - d_3 d_8 u_0 u_1 v_1 v_2 - 
  d_2 d_4 d_9 u_0 u_1 v_1 v_2
\\&\; - d_2 d_5 d_9 u_0 u_1 v_1 v_2 - 
  d_2 d_3 d_7 d_9 u_0 u_1 v_1 v_2 + d_2 d_6 d_7 d_9 u_0 u_1 v_1 v_2 - 
  d_2 d_3 d_8 d_9 u_0 u_1 v_1 v_2
\\&\; + d_2 d_6 d_8 d_9 u_0 u_1 v_1 v_2 - d_3 d_7 u_0 v_2^2 - 
  d_2 d_4 d_9 u_0 v_2^2 - d_2 d_3 d_7 d_9 u_0 v_2^2 + d_2 d_6 d_7 d_9 u_0 v_2^2
\\&\; + 
  d_2 v_3 v_4 + d_2 d_9 u_1 v_1 v_5 + d_2 d_9 v_2 v_5
 - d_2 d_9 u_1 v_1 v_6 - 
d_2 d_9 v_2 v_6
\\ \bullet&\;
 d_2 d_5 d_9 u_0 u_1^2 v_1^2 + d_2 d_3 d_8 d_9 u_0 u_1^2 v_1^2 - 
  d_2 d_6 d_8 d_9 u_0 u_1^2 v_1^2 + d_1 d_9 u_0 u_1 v_1 v_2 + 
  d_2 d_4 d_9 u_0 u_1 v_1 v_2
\\&\; + d_2 d_5 d_9 u_0 u_1 v_1 v_2 + 
  d_2 d_3 d_7 d_9 u_0 u_1 v_1 v_2 - d_2 d_6 d_7 d_9 u_0 u_1 v_1 v_2 + 
  d_2 d_3 d_8 d_9 u_0 u_1 v_1 v_2
\\&\; - d_2 d_6 d_8 d_9 u_0 u_1 v_1 v_2 + d_1 d_9 u_0 v_2^2 + 
  d_2 d_4 d_9 u_0 v_2^2 + d_2 d_3 d_7 d_9 u_0 v_2^2- d_2 d_6 d_7 d_9 u_0 v_2^2
\\&\;  + 
  d_2 u_0^2 v_1 v_4 + d_2 v_4^2 - d_2 d_9 u_1 v_1 v_5 
  - d_2 d_9 v_2 v_5 + 
  d_2 d_8 u_1 v_1 v_6 + d_2 d_7 v_2 v_6 
\\ \bullet&\; d_3 u_0^3 v_1 v_2 + d_2 u_1^2 v_1 v_2 + d_2 u_1 v_2^2 + d_2 d_3 d_9 u_0 u_1 v_1 v_4 - 
  d_2 d_6 d_9 u_0 u_1 v_1 v_4 + d_3 u_0 v_2 v_4 
\\&\;+ d_2 d_3 d_9 u_0 v_2 v_4 - 
  d_2 d_6 d_9 u_0 v_2 v_4 + d_2 v_3 v_6
  \\ \bullet&\;-d_2 d_5 d_9 u_0^3 u_1 v_1^2 + 
  d_2 d_6 d_8 d_9 u_0^3 u_1 v_1^2 - d_2 d_3 d_9^2 u_0^3 u_1 v_1^2 + 
  d_2^2 d_8 d_9 u_1^3 v_1^2 - d_2^2 d_9^2 u_1^3 v_1^2
\\&\; - d_2 d_4 d_9 u_0^3 v_1 v_2 + 
  d_2 d_6 d_7 d_9 u_0^3 v_1 v_2 - d_2 d_3 d_9^2 u_0^3 v_1 v_2 + d_2 d_8 u_1^2 v_1 v_2+   d_2^2 d_7 d_9 u_1^2 v_1 v_2
\\&\;  + d_2^2 d_8 d_9 u_1^2 v_1 v_2 - 
  2 d_2^2 d_9^2 u_1^2 v_1 v_2 + d_2 d_7 u_1 v_2^2 + d_2^2 d_7 d_9 u_1 v_2^2- 
  d_2^2 d_9^2 u_1 v_2^2 
\\&\; - d_2 d_5 d_9 u_0 u_1 v_1 v_3 + d_2 d_6 d_8 d_9 u_0 u_1 v_1 v_3 -
   d_1 d_2 d_9^2 u_0 u_1 v_1 v_3 - d_1 d_9 u_0 v_2 v_3 
\\&\; - d_2 d_4 d_9 u_0 v_2 v_3+ 
  d_2 d_6 d_7 d_9 u_0 v_2 v_3 - d_1 d_2 d_9^2 u_0 v_2 v_3 - d_2 d_5 d_9 u_0 u_1 v_1 v_4 
\\&\;- 
  d_2 d_4 d_9 u_0 v_2 v_4 + d_2 d_9 u_0^2 v_1 v_5 + d_2 d_9 v_3 v_5 - 
  d_2 d_9 u_0^2 v_1 v_6 - d_2 d_9 v_4 v_6
  \\ \bullet&\;
  -d_3 d_8 u_0^3 u_1 v_1^2 + 
  d_3 d_9 u_0^3 u_1 v_1^2 - d_2 d_8 u_1^3 v_1^2 + d_2 d_9 u_1^3 v_1^2- 
  d_3 d_7 u_0^3 v_1 v_2  + d_3 d_9 u_0^3 v_1 v_2
\\&\; - d_2 d_7 u_1^2 v_1 v_2 - 
  d_2 d_8 u_1^2 v_1 v_2 + 2 d_2 d_9 u_1^2 v_1 v_2 - d_2 d_7 u_1 v_2^2 + 
  d_2 d_9 u_1 v_2^2 - d_3 d_8 u_0 u_1 v_1 v_3
\\&\; + d_1 d_9 u_0 u_1 v_1 v_3 
- 
  d_3 d_7 u_0 v_2 v_3 + d_1 d_9 u_0 v_2 v_3 + v_4 v_5 
\\ \bullet&\; d_2 d_3 d_9 u_0^5 v_1^2 + d_2^2 d_9 u_0^2 u_1^2 v_1^2 - 
  d_2 d_3 d_5 d_8 d_9 u_0^2 u_1^2 v_1^2 + d_1 d_2 d_5 d_9^2 u_0^2 u_1^2 v_1^2 + 
  d_2^2 d_9 u_0^2 u_1 v_1 v_2 
\\&\;- d_2 d_3 d_5 d_7 d_9 u_0^2 u_1 v_1 v_2 - 
  d_2 d_3 d_4 d_8 d_9 u_0^2 u_1 v_1 v_2 + d_1 d_2 d_4 d_9^2 u_0^2 u_1 v_1 v_2 + 
  d_1 d_2 d_5 d_9^2 u_0^2 u_1 v_1 v_2
\\&\; - d_2 d_3 d_4 d_7 d_9 u_0^2 v_2^2 + 
  d_1 d_2 d_4 d_9^2 u_0^2 v_2^2 + d_1 d_2 d_9 u_0^3 v_1 v_3 + d_2^2 d_8 u_1^2 v_1 v_3+
   d_2^2 d_7 u_1 v_2 v_3
\\&\;  + d_2 d_3 d_9 u_0^3 v_1 v_4 + d_2^2 d_9 u_1^2 v_1 v_4+ 
  d_2^2 d_9 u_1 v_2 v_4 - d_2 d_5 d_9 u_0 u_1 v_1 v_5+ d_2 d_6 d_8 d_9 u_0 u_1 v_1 v_5
\\&\;   - 
  d_1 d_2 d_9^2 u_0 u_1 v_1 v_5- d_1 d_9 u_0 v_2 v_5 - d_2 d_4 d_9 u_0 v_2 v_5  + 
  d_2 d_6 d_7 d_9 u_0 v_2 v_5 - d_1 d_2 d_9^2 u_0 v_2 v_5
\\&\; + d_2 d_9 v_5^2 - 
  d_2 d_3 d_8 d_9 u_0 u_1 v_1 v_6 + d_1 d_2 d_9^2 u_0 u_1 v_1 v_6- 
  d_2 d_3 d_7 d_9 u_0 v_2 v_6 
 + d_1 d_2 d_9^2 u_0 v_2 v_6
\\ \bullet&\;
 d_2 d_3 u_0^5 v_1^2 + d_2^2 u_0^2 u_1^2 v_1^2 + 
  d_2 d_3^2 d_8 d_9 u_0^2 u_1^2 v_1^2 - d_2 d_3 d_6 d_8 d_9 u_0^2 u_1^2 v_1^2 - 
  d_1 d_2 d_3 d_9^2 u_0^2 u_1^2 v_1^2 
\\&\;+ d_1 d_2 d_6 d_9^2 u_0^2 u_1^2 v_1^2 + 
  d_2^2 u_0^2 u_1 v_1 v_2 + d_3^2 d_8 u_0^2 u_1 v_1 v_2 - 
  d_1 d_3 d_9 u_0^2 u_1 v_1 v_2 + d_2 d_3^2 d_7 d_9 u_0^2 u_1 v_1 v_2
\\&\; - 
  d_2 d_3 d_6 d_7 d_9 u_0^2 u_1 v_1 v_2 + d_2 d_3^2 d_8 d_9 u_0^2 u_1 v_1 v_2 - 
  d_2 d_3 d_6 d_8 d_9 u_0^2 u_1 v_1 v_2 - 2 d_1 d_2 d_3 d_9^2 u_0^2 u_1 v_1 v_2
\\&\; + 
  2 d_1 d_2 d_6 d_9^2 u_0^2 u_1 v_1 v_2 + d_3^2 d_7 u_0^2 v_2^2 - 
  d_1 d_3 d_9 u_0^2 v_2^2 
+ d_2 d_3^2 d_7 d_9 u_0^2 v_2^2 - 
  d_2 d_3 d_6 d_7 d_9 u_0^2 v_2^2
\\&\; - d_1 d_2 d_3 d_9^2 u_0^2 v_2^2 + 
  d_1 d_2 d_6 d_9^2 u_0^2 v_2^2 + d_2 d_3 u_0^3 v_1 v_3 + d_2^2 u_1^2 v_1 v_3 + 
  d_2^2 u_1 v_2 v_3
\\&\; + d_2 d_3 u_0^3 v_1 v_4  + d_2^2 u_1^2 v_1 v_4+ d_2^2 u_1 v_2 v_4 +
   d_2 v_5 v_6, 
\\ \bullet&\; d_2^2 d_3 d_9 u_0^5 v_1^2 + d_2^3 d_9 u_0^2 u_1^2 v_1^2 - 
  d_2^2 d_3 d_5 d_9^2 u_0^2 u_1^2 v_1^2 + d_2^2 d_5 d_6 d_9^2 u_0^2 u_1^2 v_1^2 + 
  d_2^2 d_3 d_6 d_8 d_9^2 u_0^2 u_1^2 v_1^2
\\&\; - 
  d_2^2 d_6^2 d_8 d_9^2 u_0^2 u_1^2 v_1^2 - d_1 d_2^2 d_3 d_9^3 u_0^2 u_1^2 v_1^2+
  d_1 d_2^2 d_6 d_9^3 u_0^2 u_1^2 v_1^2 + d_2^3 d_9 u_0^2 u_1 v_1 v_2 
\\&\; - 
  d_2 d_3 d_5 d_9 u_0^2 u_1 v_1 v_2 + d_2 d_3 d_6 d_8 d_9 u_0^2 u_1 v_1 v_2 - 
  2 d_1 d_2 d_3 d_9^2 u_0^2 u_1 v_1 v_2 - d_2^2 d_3 d_4 d_9^2 u_0^2 u_1 v_1 v_2
\\&\; - 
  d_2^2 d_3 d_5 d_9^2 u_0^2 u_1 v_1 v_2 
+ d_1 d_2 d_6 d_9^2 u_0^2 u_1 v_1 v_2 + 
  d_2^2 d_4 d_6 d_9^2 u_0^2 u_1 v_1 v_2 + d_2^2 d_5 d_6 d_9^2 u_0^2 u_1 v_1 v_2 
\\&\;+ 
  d_2^2 d_3 d_6 d_7 d_9^2 u_0^2 u_1 v_1 v_2 - d_2^2 d_6^2 d_7 d_9^2 u_0^2 u_1 v_1 v_2 +
   d_2^2 d_3 d_6 d_8 d_9^2 u_0^2 u_1 v_1 v_2 
\\&\;- 
  d_2^2 d_6^2 d_8 d_9^2 u_0^2 u_1 v_1 v_2 - 2 d_1 d_2^2 d_3 d_9^3 u_0^2 u_1 v_1 v_2 + 
  2 d_1 d_2^2 d_6 d_9^3 u_0^2 u_1 v_1 v_2 - d_1 d_3 d_9 u_0^2 v_2^2
  \\&\;
+ - 
  d_2 d_3 d_4 d_9 u_0^2 v_2^2 + d_2 d_3 d_6 d_7 d_9 u_0^2 v_2^2 - 
  2 d_1 d_2 d_3 d_9^2 u_0^2 v_2^2 - d_2^2 d_3 d_4 d_9^2 u_0^2 v_2^2 + 
  d_1 d_2 d_6 d_9^2 u_0^2 v_2^2
\\&\;      + d_2^2 d_4 d_6 d_9^2 u_0^2 v_2^2 + 
  d_2^2 d_3 d_6 d_7 d_9^2 u_0^2 v_2^2 - d_2^2 d_6^2 d_7 d_9^2 u_0^2 v_2^2 - 
  d_1 d_2^2 d_3 d_9^3 u_0^2 v_2^2 + d_1 d_2^2 d_6 d_9^3 u_0^2 v_2^2
\\&\; + 
  d_2^2 d_3 d_9 u_0^3 v_1 v_3 + d_2^3 d_9 u_1^2 v_1 v_3 + d_2^3 d_9 u_1 v_2 v_3 + 
  d_2^2 d_6 d_9 u_0^3 v_1 v_4 + d_2^3 d_9 u_1^2 v_1 v_4 + d_2^2 u_1 v_2 v_4
\\&\; + 
  d_2^3 d_9 u_1 v_2 v_4 + d_2^2 d_3 d_9^2 u_0 u_1 v_1 v_5 - 
  d_2^2 d_6 d_9^2 u_0 u_1 v_1 v_5 + d_2 d_3 d_9 u_0 v_2 v_5 + 
  d_2^2 d_3 d_9^2 u_0 v_2 v_5 
\\&\;- d_2^2 d_6 d_9^2 u_0 v_2 v_5 + 
d_2^2 d_5 d_9 u_0 u_1 v_1 v_6 + d_2^2 d_4 d_9 u_0 v_2 v_6 + d_2^2 d_9 v_6^2
   \end{align*}
\end{sizeddisplay}

\subsection{Equations of $\boldsymbol{Y_0}$ in the case of $\boldsymbol{(C20,p=2,\emptyset,D_3 2_7)}$}\label{Y0}

\begin{sizeddisplay}{\small}
\begin{align*}
\bullet &\;\frac{1}{32} (-63 + 259 \,\ii \sqrt{7}) u_0^4 w_1^2 + 
\frac 1{32}  (63 - 259 \,\ii \sqrt{7}) u_0 u_1^2 w_1^2 + 
 \frac 1{64}  (21 + 31 \,\ii \sqrt{7}) u_0 u_1 w_1 w_2 
 \\&
 + 
\frac 1{32}  (-1 - 3 \,\ii \sqrt{7}) u_0 w_2^2 + 
\frac  1{16}  (133 - 17 \,\ii \sqrt{7}) u_0^2 w_1 w_3 + w_3^2 
+\frac  1{32}  (-217 + 21 \,\ii \sqrt{7}) u_0^2 w_1 w_4 
\\&
+ \frac 1{64} (77 - 57 \,\ii \sqrt{7}) u_1 w_1 w_5 + 
 \frac 1{32} (-67 + 23 \,\ii \sqrt{7}) w_2 w_5 
 + 
 \frac 1{16} (-49 + 13 \,\ii \sqrt{7}) u_1 w_1 w_6 
 \\&
 + \frac 18 (9 - 5 \,\ii \sqrt{7}) w_2 w_6
\\\bullet &\; 
\frac  14 (7 + 30 \,\ii\sqrt{7}) u_0^4 w_1^2 + 
\frac    14 (-7 - 30 \,\ii\sqrt{7}) u_0 u_1^2 w_1^2 +
\frac   1{32} (11 + 9 \,\ii\sqrt{7}) u_0 u_1 w_1 w_2
\\&
 + 
 \frac   1{56} (-7 - 5 \,\ii\sqrt{7}) u_0 w_2^2 + 
 \frac   14 (23 - 4 \,\ii\sqrt{7}) u_0^2 w_1 w_3 + 
 \frac   1{16} (-77 + 9 \,\ii\sqrt{7}) u_0^2 w_1 w_4 + w_3 w_4
 \\&
  + 
 \frac   1{32} (51 - 7 \,\ii\sqrt{7}) u_1 w_1 w_5 + 
 \frac   1{56} (-77 + 41 \,\ii\sqrt{7}) w_2 w_5+ 
 \frac   18 (-23 + 3 \,\ii\sqrt{7}) u_1 w_1 w_6 
 \\& + 
 \frac  1{14} (7 - 9 \,\ii\sqrt{7}) w_2 w_6
 \\\bullet &\; 
 \frac  1{64} (275 + 471 \,\ii\sqrt{7}) u_0^4 w_1^2 + 
 \frac   1{64} (-275 - 471 \,\ii\sqrt{7}) u_0 u_1^2 w_1^2 + 
 \frac   1{448} (161 - 11 \,\ii\sqrt{7}) u_0 u_1 w_1 w_2
 \\& + 
 \frac   1{28} (-7 - 2 \,\ii\sqrt{7}) u_0 w_2^2 + 
 \frac   1{64} (229 + 17 \,\ii\sqrt{7}) u_0^2 w_1 w_3 + 
 \frac   14 (-10 - 3 \,\ii\sqrt{7}) u_0^2 w_1 w_4 + w_4^2
 \\& + 
 \frac   1{112} (175 - 31 \,\ii\sqrt{7}) u_1 w_1 w_5 + 
 \frac   1{28} (-21 + 22 \,\ii\sqrt{7}) w_2 w_5 + 
 \frac   1{14} (-42 + 11 \,\ii\sqrt{7}) u_1 w_1 w_6 
 \\&
 - 
\frac  57  \,\ii \sqrt{7} w_2 w_6
\\\bullet &\;
 \frac 1{1472}(2107 - 567 \,\ii\sqrt{7}) u_0^3 u_1 w_1^2
+\frac 1{1472} (-2107 + 567 \,\ii\sqrt{7}) u_1^3 w_1^2+ 
\frac 1{1472}(-973 + 617 \,\ii\sqrt{7}) u_0^3 w_1 w_2 
 \\&+
\frac  1{736} (427 - 607 \,\ii\sqrt{7}) u_1^2 w_1 w_2 + 
\frac  1{736} ( -177 + 73 \,\ii\sqrt{7}) u_1 w_2^2 + 
\frac   1{16} (-63 + 47 \,\ii\sqrt{7}) u_0 u_1 w_1 w_3 
 \\&+
\frac     1{184} (254 - 29 \,\ii\sqrt{7}) u_0 w_2 w_3 + 
\frac     1{184} (861 - 497 \,\ii\sqrt{7}) u_0 u_1 w_1 w_4 + 
 \frac    1{736} (-1071 + 147 \,\ii\sqrt{7}) u_0 w_2 w_4  
 \\&+
  \frac  1{64} (161 - 61 \,\ii\sqrt{7}) u_0^2 w_1 w_5 + 
  \frac  1{368} (-1169 + 413 \,\ii\sqrt{7}) u_0^2 w_1 w_6 + w_3 w_6 + 
   \frac 1{92} (-63 - 13 \,\ii\sqrt{7}) w_4 w_6
   \\\bullet &\;
  \frac 1{736} (2457 - 735 \,\ii\sqrt{7}) u_0^3 u_1 w_1^2 + 
  \frac  1{736} (-2457 + 735 \,\ii\sqrt{7}) u_1^3 w_1^2 + 
   \frac 1{368} (-413 + 201 \,\ii\sqrt{7}) u_0^3 w_1 w_2  
 \\&+
   \frac 1{92} (119 - 70 \,\ii\sqrt{7}) u_1^2 w_1 w_2 + 
   \frac 1{736} (-257 + 139 \,\ii\sqrt{7}) u_1 w_2^2 + 
   \frac 18 (-49 + 23 \,\ii\sqrt{7}) u_0 u_1 w_1 w_3 
 \\&+
   \frac 1{736} (1565 - 167 \,\ii\sqrt{7}) u_0 w_2 w_3 + 
   \frac 1{92} (651 - 203 \,\ii\sqrt{7}) u_0 u_1 w_1 w_4 + 
   \frac 1{184} (-399 + 71 \,\ii\sqrt{7}) u_0 w_2 w_4 
 \\&+
   \frac 1{16} (105 - 37 \,\ii\sqrt{7}) u_0^2 w_1 w_5 + w_3 w_5 + 
   \frac 1{92} (-707 + 171 \,\ii\sqrt{7}) u_0^2 w_1 w_6 + 
   \frac 1{23} (-14 - 8 \,\ii\sqrt{7}) w_4 w_6 
 \\\bullet &\; 
   \frac 1{5888}(4249 -   171 \,\ii\sqrt{7}) u_0^3 u_1 w_1^2
     +
   \frac 1{5888} (-4249 + 171 \,\ii\sqrt{7}) u_1^3 w_1^2
   +  \frac 1{2944} (-1551 + 589 \,\ii\sqrt{7}) u_0^3 w_1 w_2 
 \\&+
  \frac 1{368} (371 - 41 \,\ii\sqrt{7}) u_1^2 w_1 w_2 
  +  \frac 1{5888} (-903 +   565 \,\ii\sqrt{7}) u_1 w_2^2 + 
  \frac 1{128} (-399 + 77 \,\ii\sqrt{7}) u_0 u_1 w_1 w_3 
 \\&
  +  \frac 1{5888} (6223 -      653 \,\ii\sqrt{7}) u_0 w_2 w_3+ 
  \frac 1{184} (707 - 33 \,\ii\sqrt{7}) u_0 u_1 w_1 w_4 + 
   \frac 1{736} (-765 + 151 \,\ii\sqrt{7}) u_0 w_2 w_4 
 \\&+
  \frac  1{16} (57 - 27 \,\ii\sqrt{7}) u_0^2 w_1 w_5 + w_4 w_5 + 
  \frac  1{368} (-1755 + 433 \,\ii\sqrt{7}) u_0^2 w_1 w_6 + 
   \frac 1{184} (-159 - 35 \,\ii\sqrt{7}) w_4 w_6 
 \\\bullet &\; 
  \frac 1{32} (245 + 217 \,\ii\sqrt{7}) u_0^5 w_1^2 + 
   \frac 1{32} (-245 - 217 \,\ii\sqrt{7}) u_0^2 u_1^2 w_1^2 + 
   \frac 1{256} (287 - 157 \,\ii\sqrt{7}) u_0^2 u_1 w_1 w_2 
 \\&
+   \frac 1{256} (-71 - 11 \,\ii\sqrt{7}) u_0^2 w_2^2 + 
  \frac  1{32} (119 - 21 \,\ii\sqrt{7}) u_0^3 w_1 w_3 + 
   \frac 1{256} (-7 + 181 \,\ii\sqrt{7}) u_1^2 w_1 w_3 
 \\&
 + \frac  1{256} (-121 - 53 \,\ii\sqrt{7}) u_1 w_2 w_3 + 
 \frac   18 (-49 - 7 \,\ii\sqrt{7}) u_0^3 w_1 w_4 + 
  \frac  1{32} (119 - 21 \,\ii\sqrt{7}) u_1^2 w_1 w_4 
 \\&
+  \frac  1{32} (-7 + 5 \,\ii\sqrt{7}) u_1 w_2 w_4 + 
  \frac  1{32} (-343 + 17 \,\ii\sqrt{7}) u_0 u_1 w_1 w_5 + 
  \frac  1{32} (33 + 29 \,\ii\sqrt{7}) u_0 w_2 w_5 + w_5^2
 \\&
 +
   \frac 1{16} (133 + 25 \,\ii\sqrt{7}) u_0 u_1 w_1 w_6 + 
  \frac  1{16} (-17 - 13 \,\ii\sqrt{7}) u_0 w_2 w_6 
 \\\bullet &\; 
  \frac 1{32} (245 + 217 \,\ii\sqrt{7}) u_0^5 w_1^2 + 
   \frac 1{32} (-245 - 217 \,\ii\sqrt{7}) u_0^2 u_1^2 w_1^2 + 
   \frac 1{256} (287 - 157 \,\ii\sqrt{7}) u_0^2 u_1 w_1 w_2 
 \\&
+   \frac 1{256} (-71 - 11 \,\ii\sqrt{7}) u_0^2 w_2^2 + 
  \frac  1{32} (119 - 21 \,\ii\sqrt{7}) u_0^3 w_1 w_3 + 
   \frac 1{256} (-7 + 181 \,\ii\sqrt{7}) u_1^2 w_1 w_3 
 \\&
 + \frac  1{256} (-121 - 53 \,\ii\sqrt{7}) u_1 w_2 w_3 + 
 \frac   18 (-49 - 7 \,\ii\sqrt{7}) u_0^3 w_1 w_4 + 
  \frac  1{32} (119 - 21 \,\ii\sqrt{7}) u_1^2 w_1 w_4 
 \\&
+  \frac  1{32} (-7 + 5 \,\ii\sqrt{7}) u_1 w_2 w_4 + 
  \frac  1{32} (-343 + 17 \,\ii\sqrt{7}) u_0 u_1 w_1 w_5 + 
  \frac  1{32} (33 + 29 \,\ii\sqrt{7}) u_0 w_2 w_5 + w_5^2
 \\&
 +
   \frac 1{16} (133 + 25 \,\ii\sqrt{7}) u_0 u_1 w_1 w_6 + 
  \frac  1{16} (-17 - 13 \,\ii\sqrt{7}) u_0 w_2 w_6
  \\\bullet &\;    \frac 1{16} (98 + 119 \,\ii\sqrt{7}) u_0^5 w_1^2 +
  \frac  1{16} (-98 - 119 \,\ii\sqrt{7}) u_0^2 u_1^2 w_1^2 + 
   \frac 1{32} (28 - 29 \,\ii\sqrt{7}) u_0^2 u_1 w_1 w_2
 \\& + 
   \frac 1{64} (-27 + \,\ii\sqrt{7}) u_0^2 w_2^2 +
   \frac 1{32} (133 + 77 \,\ii\sqrt{7}) u_0^3 w_1 w_3 + 
  \frac  1{64} (-49 + 19 \,\ii\sqrt{7}) u_1^2 w_1 w_3
 \\& + 
  \frac  1{64} (11 - 17 \,\ii\sqrt{7}) u_1 w_2 w_3 +
   \frac 1{32} (-147 - 105 \,\ii\sqrt{7}) u_0^3 w_1 w_4 + 
   \frac 18 (28 - 7 \,\ii\sqrt{7}) u_1^2 w_1 w_4
 \\& + 
   \frac 18 (-7 + 2 \,\ii\sqrt{7}) u_1 w_2 w_4 + 
  \frac  1{64} (-651 - 17 \,\ii\sqrt{7}) u_0 u_1 w_1 w_5 + 
   \frac 1{16} (19 + 13 \,\ii\sqrt{7}) u_0 w_2 w_5
 \\& + 
   \frac 18 (56 + 21 \,\ii\sqrt{7}) u_0 u_1 w_1 w_6 + 
  \frac  1{16} (-17 - 13 \,\ii\sqrt{7}) u_0 w_2 w_6 + w_5 w_6, 
  \\\bullet &\;  \frac 1{128} (735 + 1057 \,\ii\sqrt{7}) u_0^5 w_1^2 + 
  \frac  1{128} (-735 - 1057 \,\ii\sqrt{7}) u_0^2 u_1^2 w_1^2 +
  \frac  1{256} (273 - 333 \,\ii\sqrt{7}) u_0^2 u_1 w_1 w_2
 \\& + 
  \frac  1{32} (-19 + 2 \,\ii\sqrt{7}) u_0^2 w_2^2 + 
  \frac  1{32} (56 + 175 \,\ii\sqrt{7}) u_0^3 w_1 w_3 + 
   \frac 1{64} (-63 + 13 \,\ii\sqrt{7}) u_1^2 w_1 w_3
 \\& + 
   \frac 1{64} (59 - 21 \,\ii\sqrt{7}) u_1 w_2 w_3 + 
   \frac 1{128} (49 - 749 \,\ii\sqrt{7}) u_0^3 w_1 w_4 + 
   \frac 1{64} (161 - 77 \,\ii\sqrt{7}) u_1^2 w_1 w_4 
 \\&+ 
  \frac  1{64} (-105 + 21 \,\ii\sqrt{7}) u_1 w_2 w_4 + 
 \frac   1{256} (-2471 - 405 \,\ii\sqrt{7}) u_0 u_1 w_1 w_5 + 
 \frac   18 (9 + 7 \,\ii\sqrt{7}) u_0 w_2 w_5
 \\& + 
   \frac 1{64} (315 + 273 \,\ii\sqrt{7}) u_0 u_1 w_1 w_6 
 \frac   1{16} (-13 - 15 \,\ii\sqrt{7}) u_0 w_2 w_6 + w_6^2
\end{align*}
\end{sizeddisplay}

\subsection{Parametric equation of $\boldsymbol{S_2}$ in $\boldsymbol{Y_0}$  in the case of $\boldsymbol{(C20,p=2,\emptyset,D_3 2_7)}$}\label{S2param}

\begin{sizeddisplay}{\small}
\begin{align*}
u_0&=
\tfrac 12 \big(5 + 9 \,\ii\sqrt 7
-16  (1 +  \,\ii  \sqrt 7) t
+ (13 + 7 \,\ii\sqrt 7) t^2\big)
\\ u_1&=
\tfrac 12 \big(-47 + 29\,\ii\sqrt{7}+ 
     84 (1 - \,\ii\sqrt{7}) t  
     + (-14 + 86 \,\ii\sqrt{7}) t^2
     + (-21 - 31 \,\ii\sqrt{7}) t^3\big)
\\w_1&=
     1,
\\w_2&=
\tfrac 12 \big(3121 - 403\,\ii\sqrt{7}
+ (-9820 + 1052 \,\ii\sqrt{7}) t 
+ 4 (2545 - 221 \,\ii\sqrt{7}) t^2
+ (-3473 + 237 \,\ii\sqrt{7}) t^3\big)
\\w_3&=
\tfrac 12 \big(3505 - 2963\,\ii\sqrt{7}
+ 8 (-2259 + 1391 \,\ii \sqrt 7) t
+ 16 (2039 - 955 \,\ii\sqrt{7}) t^2
+ (-24950 + 9054 \,\ii\sqrt{7}) t^3 
\\&\hphantom{=\;} 
+      7 (985 - 277 \,\ii\sqrt{7}) t^4\big)
\\w_4&=
4 (-1 + t)
\big(-108 + 316\,\ii\sqrt{7}
+ 32 (25 - 27 \,\ii\sqrt{7}) t
+ (-1239 + 731 \,\ii\sqrt{7}) t^2
+ 8 (67 - 23 \,\ii\sqrt{7}) t^3\big)
 \\w_5&=
 \tfrac 12 \big(2265 - 2539\,\ii\sqrt{7}
 + 8  (-2941  + 1523 \,\ii \sqrt 7 ) t
 + 736 (99 - 29 \,\ii\sqrt{7}) t^2 
 \\&\hphantom{=;} 
 + (-99624 + 16264 \,\ii\sqrt{7}) t^3  
 + 48 (1323 - 95 \,\ii\sqrt{7}) t^4
 + (-15479 - 5 \,\ii\sqrt{7}) t^5\big)
\\w_6&=
\tfrac 12 (-1 + t)
\big(4720 + 336\,\ii\sqrt{7}- 32 (313 + 43 \,\ii \sqrt 7 ) t +(-952 - 40 \,\ii\sqrt{7}) t^2
\\&\hphantom{=;} 
+ 8 (1571 + 393 \,\ii\sqrt{7}) t^3
+ (-6325 - 2047 \,\ii\sqrt{7}) t^4\big)
\end{align*}
\end{sizeddisplay}

\subsection{Parametric equation of $\boldsymbol{S_1}$ in $\boldsymbol{Y_0}$  in the case of $\boldsymbol{(C20,p=2,\emptyset,D_3 2_7)}$}\label{S1param}

\begin{sizeddisplay}{\small}
\begin{align*}
u_0&=
\frac 12
\big(5 + 9 \,\ii\sqrt 7
- 16  (1 +  \,\ii  \sqrt 7) t
+ (13 + 7 \,\ii\sqrt 7) t^2\big)
\\ u_1&=
\frac 12 \big(-47 + 29 \,\ii\sqrt 7
+  84 (1 - \,\ii\sqrt 7) t 
+ (-14 + 86 \,\ii\sqrt 7) t^2
+ (-21 - 31 \,\ii\sqrt 7) t^3\big)
\\ w_1&=     1 
\\ w_2&=
\frac 12 (1049 - 379 \,\ii\sqrt 7)
+ (-908 + 564 \,\ii\sqrt 7) t 
+ (357 - 529 \,\ii\sqrt 7) t^2
+ 2 (15 + 77 \,\ii\sqrt 7) t^3 
\\ w_3&=
\frac 12\big(-(1271 - 363 \,\ii\sqrt 7)
+  16 (193 + 37 \,\ii\sqrt 7) t
+ (-5446 - 530 \,\ii\sqrt 7) t^2
+ 4096 t^3 + \frac 14 (-4417  + 477  \,\ii \sqrt 7) t^4\big) 
\\ w_4&=
-1105 - 77 \,\ii\sqrt 7
+ 16 (261 + 7 \,\ii\sqrt 7) t 
+ (-6096 + 208 \,\ii\sqrt 7) t^2
+  4 (997 - 113 \,\ii\sqrt 7) t^3
+ (-965 + 209 \,\ii\sqrt 7) t^4 
\\ w_5&=
2 (-1 + t)
\big(-112 + 48 \,\ii\sqrt 7
+ 8 (71 + 11 \,\ii\sqrt 7) t
+ (-944 - 400 \,\ii\sqrt 7) t^2
+ 64 (13 + 7 \,\ii\sqrt 7) t^3
+ (-369 - 179 \,\ii\sqrt 7) t^4\big) 
\\ w_6&=
\frac  14 (-1 + t)
\big(16 (75 + \,\ii\sqrt 7)
+ 64 (21 + 47 \,\ii\sqrt 7) t
+ (-10696 - 6232 \,\ii\sqrt 7) t^2
+ 48 (275 + 89 \,\ii\sqrt 7) t^3  
\\&\hphantom{=;} 
+     (-5197 - 967 \,\ii\sqrt 7) t^4\big)
\end{align*}
\end{sizeddisplay}

\subsection{(Birational) automorphism of order 3 of $\boldsymbol{Y_0}$  for $\boldsymbol{(C20,p=2,\emptyset,D_3 2_7)}$}\label{Auto}

\begin{sizeddisplay}{\small}
\begin{align*}
  u_0&\longmapsto u_0
  \\ u_1&\longmapsto u_1
  \\ w_1&\longmapsto
\frac 1{1792 (u_0^3 - u_1^2)}
\big(7 (383 - 29 \,\ii\sqrt 7) u_0^4 w_1
+ u_0 u_1
(7  (-331 + \,\ii\sqrt 7) u_1 w_1 
        + (-91 + 177 \,\ii\sqrt 7) w_2)  
\\&\hphantom{\longmapsto+}
+ 7 u_0^2
(w_3 - 99 \,\ii\sqrt 7 w_3 + 40(-1 + 3 \,\ii\sqrt 7) w_4) 
+ 16 u_1 ((77 + 17 \,\ii\sqrt 7) w_5
+ (-77 - 25 \,\ii\sqrt 7) w_6)\big)
\\ w_2&\longmapsto
    \frac 1{256 (u_0^3 - u_1^2)}
    \big(u_0^4 ((2471 - 405 \,\ii\sqrt 7) u_1 w_1
    +  16 (3 + 7 \,\ii\sqrt 7) w_2)
     \\&\hphantom{\longmapsto+} 
       + 
    u_0 u_1^2 ((-2387 + 281 \,\ii\sqrt 7) u_1 w_1 + (-325 + 
          47 \,\ii\sqrt 7) w_2) 
          \\&\hphantom{\longmapsto+}+ 
    u_0^2 u_1 ((1961 - 603 \,\ii\sqrt 7) w_3 + 
       8 (-301 + 71 \,\ii\sqrt 7) w_4) + 
    128 u_0^3 ( (-1 +  \,\ii\sqrt 7) w_5 
    \\&\hphantom{\longmapsto+}+ (6 - 2 \,\ii\sqrt 7) w_6) + 
    16 u_1^2 ((19 + 15 \,\ii\sqrt 7) w_5 + (-19 - 7 \,\ii\sqrt 7) w_6)\big) 
 \\ w_3&\longmapsto \frac 1{ 256 (u_0^3 - u_1^2)}
 \big((287 + 131 \,\ii\sqrt 7) u_0^6 w_1 + 
    32 u_1^3 ((35 + 11 \,\ii\sqrt 7) u_1 w_1 + (-5 - \,\ii\sqrt 7) w_2)
 \\& \hphantom{\longmapsto+}+ 
    u_0^3 u_1 ((-1547 - 447 \,\ii\sqrt 7) u_1 w_1 + (259 + 
          87 \,\ii\sqrt 7) w_2)
    \\&\hphantom{\longmapsto+} + 
    16 u_0 u_1^2 ((-45 + 7 \,\ii\sqrt 7) w_3 + 8 (7 - \,\ii\sqrt 7) w_4) + 
    u_0^4 ((673 - 579 \,\ii\sqrt 7) w_3
    \\& \hphantom{\longmapsto+}+ 8 (-133 + 79 \,\ii\sqrt 7) w_4) + 
    16 u_0^2 u_1 ((11 + 7 \,\ii\sqrt 7) w_5 + 5 (w_6 - 3 \,\ii\sqrt 7 w_6))\big) 
 \\ w_4&\longmapsto  
 \frac 1{896 (u_0^3 - u_1^2)}
 \big(14 (39 + 43 \,\ii\sqrt 7) u_0^6 w_1 + 
    14 u_1^3 ((291 + 55 \,\ii\sqrt 7) u_1 w_1 + 3 (-9 - 5 \,\ii\sqrt 7) w_2)
    \\ &\hphantom{\longmapsto+}+ 
    u_0^3 u_1 (56 (-89 - 21 \,\ii\sqrt 7) u_1 w_1 + (693 + 
          257 \,\ii\sqrt 7) w_2) + 
    14 u_0^4 ((185 - 139 \,\ii\sqrt 7) w_3 
    \\&\hphantom{\longmapsto+}+ 8 (-33 + 19 \,\ii\sqrt 7) w_4) + 
    7 u_0 u_1^2 ((-467 + 153 \,\ii\sqrt 7) w_3 + 
       8 (71 - 21 \,\ii\sqrt 7) w_4)
       \\&\hphantom{\longmapsto+} + 
       16 u_0^2 u_1 ((7 + 11 \,\ii\sqrt 7) w_5 + (21 - 31 \,\ii\sqrt 7) w_6)\big)
\\ w_5&\longmapsto       
  \frac 1{256} \big((77 + 57 \,\ii\sqrt 7) u_0^2 u_1 w_1 + 
     16 u_1 ((-17 - 13 \,\ii\sqrt 7) w_3 + 16 \,\ii\sqrt 7 w_4)
  \\&\hphantom{\longmapsto+} + 
    \frac {32}{
     u_0^3 - u_1^2}
    u_0^3 (u_0^2 (w_2 + \,\ii\sqrt 7 w_2) + (-3 - 7 \,\ii\sqrt 7) u_1 w_3 + 
        8 \,\ii\sqrt 7
          u_1 w_4
     \\&\hphantom{\longmapsto+} + (6 + 2 \,\ii\sqrt 7) u_0 w_5 + (-4 - 
           4 \,\ii\sqrt 7) u_0 w_6) + 
     u_0 ((11 - 65 \,\ii\sqrt 7) u_0 w_2 
           \\&\hphantom{\longmapsto+} + 
        16 ((-3 + 9 \,\ii\sqrt 7) w_5+ (15 - 13 \,\ii\sqrt 7) w_6))\big)
\\ w_6&\longmapsto      
        \frac 1{32} \big(u_0^2 ((7 + 11 \,\ii\sqrt 7) u_1 w_1
        + (-5 - 
           5 \,\ii\sqrt 7) w_2) + 
           4 u_1 (-4  (4 + \,\ii\sqrt 7) w_3 + (7 + 5 \,\ii\sqrt 7) w_4)
           \\&\hphantom{\longmapsto+} + 
     6 (7 + 5 \,\ii\sqrt 7) u_0 w_5 + 8 (-3 - 5 \,\ii\sqrt 7) u_0 w_6\big)
\end{align*}
\end{sizeddisplay}

\subsection{Equations of $\boldsymbol{Y_0}$ in the case of Keum's fake projective plane}\label{Y0Keum}

\begin{sizeddisplay}{\small}
\begin{align*}
\bullet &\;
\frac 1{4096}(4439 -  677 \ii  \sqrt{7})( u_0^3  -  u_1^2) u_0 w_1^2 
 + 
  \frac 1{256}(-117 - 417 \ii \sqrt{7}) u_0 u_1 w_1 w_2 + 
  \frac 1{16}(-27 - 9 \ii \sqrt{7}) u_0 w_2^2
 \\&
  +\frac 1{128} (-193 +  3 \ii \sqrt{7}) u_0^2 w_1 w_3
   + w_3^2 
   + 
 \frac 1{256}  (119 + 315 \ii \sqrt{7}) u_0^2 w_1 w_4 + \frac 1{32}
  (-15 -  7 \ii \sqrt{7}) u_1 w_1 w_5
\\&
 + \frac 1{16}
  (-3 +  15 \ii \sqrt{7}) w_2 w_5
   + 
  \frac 1{32}(7 + 43\ii \sqrt{7}) u_1 w_1 w_6 
  + \frac 18(21 - 9 \ii \sqrt{7}) w_2 w_6
\\\bullet &\;  
 \frac 1{2048}(169 + 165\ii\sqrt{7})(u_0^3 - u_1^2)u_0 w_1^2
 + 
 \frac 1{128}(9 + 21 \ii \sqrt{7}) u_0 u_1 w_1 w_2 
 \\&
 + 
 \frac 1{16}(-9 -  27 \ii  \sqrt{7}) u_0 w_2^2
 + 
 \frac 1{64}(5 + \ii \sqrt{7}) u_0^2 w_1 w_3 
 + 
 \frac 1{128} (-331 + 33 \ii  \sqrt{7}) u_0^2 w_1 w_4 
 + w_3 w_4 
 \\&
 + \frac 18 u_1 w_1 w_5 
 + 
\frac 1{16}  (-33 - 3 \ii  \sqrt{7}) w_2 w_5 
+\frac 1{16}
 (-27 - 7 \ii \sqrt{7}) u_1 w_1 w_6
 + \frac 18
  (45 +  15 \ii \sqrt{7}) w_2 w_6 
 \\\bullet &\; 
  \frac 1{1024}(-57 + 11\ii \sqrt{7}) (u_0^3-u_1^2)u_0 w_1^2 
  + 
\frac 1{64}  (-39 - 3\ii \sqrt{7}) u_0 u_1 w_1 w_2 
+\frac 1{16} (45 - 9\ii\sqrt{7}) u_0 w_2^2
 \\&
 + \frac 1{32}(-3 + \ii \sqrt{7}) u_0^2 w_1 w_3
 + 
 \frac 1{64} (-39 - 27 \ii\sqrt{7}) u_0^2 w_1 w_4 + w_4^2 
 +
 \frac 1{16} 
  (-1 + \ii \sqrt{7}) u_1 w_1 w_5 
  \\&
  +
  \frac 1{16} 
  (-3 - 9 \ii \sqrt{7}) w_2 w_5
  + 
 \frac 18 (9 - 3 \ii \sqrt{7}) u_1 w_1 w_6 
 +
 \frac 18 (-27 + 15 \ii \sqrt{7}) w_2 w_6
\\\bullet &\; 
 \frac 1{4096}
 (-161 + 67 \ii \sqrt{7}) u_0^3 u_1 w_1^2
 + 
\frac 1{4096}  (161 - 67 \ii \sqrt{7}) u_1^3 w_1^2 
+
\frac 1{256} (93 + 9 \ii \sqrt{7}) u_0^3 w_1 w_2
\\&
 + 
 \frac
 1{512}(-63 + 93 \ii \sqrt{7}) u_1^2 w_1 w_2
 + 
 \frac 1{128} (-81 - 45 \ii \sqrt{7}) u_1 w_2^2
 + 
 \frac 1{256}((31 + 3 \ii \sqrt{7}) u_0 u_1 w_1 w_3
 \\&
 + 
 \frac 1{64}(-33 + 3 \ii \sqrt{7}) u_0 w_2 w_3 
 +
 \frac 1{512} 
  (-233 + 27 \ii \sqrt{7}) u_0 u_1 w_1 w_4
  + 
  \frac 1{128}(51 - 57 \ii \sqrt{7}) u_0 w_2 w_4 
  \\&
  +
  \frac 1{64} (11 - \ii \sqrt{7}) u_0^2 w_1 w_5 
  +
  \frac 1{64} 
  (-69 - 17 \ii \sqrt{7}) u_0^2 w_1 w_6
  + w_3 w_6 
  + 
\frac 14  (3 + 3 \ii \sqrt{7}) w_4 w_6
\\\bullet &\;
\frac 1{4096}(-1663 + 285 \ii \sqrt{7}) (u_0^3-u_1^2) u_1 w_1^2
+ 
\frac 1{256}  (-477 - 9 \ii \sqrt{7}) u_0^3 w_1 w_2
\\&
+ 
\frac 1{512}  (-561 + 243 \ii \sqrt{7}) u_1^2 w_1 w_2
 +
 \frac 1{32} (9 - 27 \ii \sqrt{7}) u_1 w_2^2
  + 
  \frac 1{256}(121 + 53 \ii \sqrt{7}) u_0 u_1 w_1 w_3
\\&
   + 
\frac 1{16}  (-3 + 9 \ii \sqrt{7}) u_0 w_2 w_3 
+ 
\frac 1{512}  (-1911 - 251 \ii  \sqrt{7}) u_0 u_1 w_1 w_4 
+ 
\frac 1{32}  (21 - 63 \ii \sqrt{7}) u_0 w_2 w_4 
\\&
+\frac 1{64} (-29 + 23 \ii \sqrt{7}) u_0^2 w_1 w_5 
   + w_3 w_5 
  + \frac 1{64}(203 - 65 \ii \sqrt{7}) u_0^2 w_1 w_6
  + \frac 14(7 + 11 \ii \sqrt{7}) w_4 w_6
\\\bullet &\;  
\frac 1{256}   (-13 + 7 \ii \sqrt{7}) (u_0^3-u_1^2) u_1 w_1^2 
 +
 \frac 18 ((-3 - 3 \ii \sqrt{7}) u_0^3 w_1 w_2
  +
\frac 1{256} (147 + 39 \ii \sqrt{7}) u_1^2 w_1 w_2
\\&
+ 
\frac 1{64}  (-99 + 9 \ii \sqrt{7}) u_1 w_2^2
 +
 \frac 1{128}(-9 - 5 \ii \sqrt{7}) u_0 u_1 w_1 w_3
  +
  \frac 1{32} (-3 + 9 \ii \sqrt{7}) u_0 w_2 w_3
\\&
   + 
\frac 1{256}  (-17 - 109 \ii \sqrt{7}) u_0 u_1 w_1 w_4
 + 
\frac 1{64}  (105 - 27 \ii \sqrt{7}) u_0 w_2 w_4
 +
 \frac 1{32} (-5 - \ii \sqrt{7}) u_0^2 w_1 w_5 
 \\&
 + 
  w_4 w_5 
  +
  \frac 1{32}(19 + 7 \ii \sqrt{7}) u_0^2 w_1 w_6 
  +
  \frac 12 (-3 + \ii \sqrt{7}) w_4 w_6
 \\\bullet &\; 
 \frac 1{32768}(5579 + 3199 \ii  \sqrt{7}) (u_0^3-u_1^2)u_0^2 w_1^2
 +
\frac 1{4096}  (12009 - 1563 \ii \sqrt{7}) u_0^2 u_1 w_1 w_2
 \\&
 + 
\frac 1{2048}  (-3897 + 3339 \ii\sqrt{7}) u_0^2 w_2^2
 + 
\frac 1{4096}  (-1267 - 7 \ii \sqrt{7}) u_0^3 w_1 w_3
\\&
+
\frac 1{4096} 
(317 - 1175 \ii \sqrt{7}) u_1^2 w_1 w_3
 + 
\frac 1{1024}  (843 + 1023 \ii \sqrt{7}) u_1 w_2 w_3
\\&
 + 
\frac 1{4096}  (49 - 1267 \ii \sqrt{7}) u_0^3 w_1 w_4
 + 
 \frac 1{1024} (1267 + 7 \ii\sqrt{7}) u_1^2 w_1 w_4
 + 
 \frac 1{2048} (-8589 + 903 \ii \sqrt{7}) u_1 w_2 w_4
\\&
  + 
  \frac 1{512}(141 - 135 \ii \sqrt{7}) u_0 u_1 w_1 w_5
   + 
  \frac 1{256}
  (387 + 87 \ii \sqrt{7}) u_0 w_2 w_5 + w_5^2 
\\&
  + 
 \frac 1{512} (-1043 + 345 \ii \sqrt{7}) u_0 u_1 w_1 w_6
 + 
\frac 1{256}  (147 - 729 \ii\sqrt{7}) u_0 w_2 w_6 
\\\bullet &\; 
\frac 1{32768} (623 + 275 \ii \sqrt{7}) (u_0^3-u_1^2)u_0^2 w_1^2
+ 
\frac 1{4096}  (1227 - 129 \ii \sqrt{7}) u_0^2 u_1 w_1 w_2 
\\&
+ 
\frac 1{2048}  (-1359 + 909 \ii \sqrt{7}) u_0^2 w_2^2 
+ 
\frac 1{4096}  (315 + 47 \ii \sqrt{7}) u_0^3 w_1 w_3
 + 
\frac 1{4096}  (-221 - 137 \ii\sqrt{7}) u_1^2 w_1 w_3 
\\&
+ 
\frac 1{1024}  (573 + 105 \ii \sqrt{7}) u_1 w_2 w_3
 + 
\frac 1{4096}  (-329 + 315 \ii\sqrt{7}) u_0^3 w_1 w_4
 + 
\frac 1{256}  (31 + 3 \ii \sqrt{7}) u_1^2 w_1 w_4
\\&
 + 
\frac 1{2048}  (-1851 + 465 \ii \sqrt{7}) u_1 w_2 w_4
+ 
\frac 1{512}  (-5 - 17 \ii\sqrt{7}) u_0 u_1 w_1 w_5
+
\frac 1{256} (-39 + 21 \ii \sqrt{7}) u_0 w_2 w_5
\\& +
 \frac 1{512} ((-213 + 95 \ii \sqrt{7}) u_0 u_1 w_1 w_6+ 
 \frac 1{256} (57 - 267 \ii \sqrt{7}) u_0 w_2 w_6 + w_5 w_6
\\\bullet &\;  
 \frac 1{32768}(623 + 275 \ii\sqrt{7}) (u_0^3-u_1^2)u_0^2 w_1^2 
 + 
 \frac 1{4096} (1227 - 129 \ii\sqrt{7}) u_0^2 u_1 w_1 w_2
\\&
 + 
 \frac 1{2048} (135 + 459 \ii \sqrt{7}) u_0^2 w_2^2 
 + 
\frac 1{4096}  (161 - 67 \ii \sqrt{7}) u_0^3 w_1 w_3 
+ 
\frac 1{4096}  (-67 - 23 \ii \sqrt{7}) u_1^2 w_1 w_3
\\&
 + 
\frac 1{1024}  (171 - 33 \ii\sqrt{7}) u_1 w_2 w_3
 + 
\frac 1{4096}  (147 + 295 \ii \sqrt{7}) u_0^3 w_1 w_4
 + 
 \frac 1{1024} (5 + 17 \ii \sqrt{7}) u_1^2 w_1 w_4
 \\&
  + 
 \frac 1{2048} (-141 + 135 \ii \sqrt{7}) u_1 w_2 w_4
  + 
\frac 1{512}  (-5 - 17 \ii \sqrt{7}) u_0 u_1 w_1 w_5
+ 
\frac 1{256}(27 + 15 \ii \sqrt{7}) u_0 w_2 w_5
 \\&
 +
 \frac 1{512} (-213 + 95 \ii \sqrt{7}) u_0 u_1 w_1 w_6
 + 
 \frac 1{256} (-309 - 129 \ii \sqrt{7}) u_0 w_2 w_6 + w_6^2
\end{align*}
\end{sizeddisplay}


\end{document}